\documentclass[12pt]{iopart}
\usepackage{iopams,amssymb,amsthm}
\usepackage{multicol}
\usepackage{color}
\usepackage{graphicx}
\usepackage{epstopdf}
\usepackage{hyperref}

\newtheorem{theorem}{Theorem}
\newtheorem{lemma}{Lemma}
\newtheorem{mynote}{Note}[section]
\theoremstyle{definition}
\newtheorem{definition}{Definition}
\newtheorem{remark}{Remark}
\newcommand{\go}{\stackrel{\circ }{\mathfrak{g}}}
\newcommand{\ao}{\stackrel{\circ }{\mathfrak{a}}}
\newcommand{\co}[1]{\stackrel{\circ }{#1}}

\newcommand{\gf}{\mathfrak{g}}
\newcommand{\af}{\mathfrak{a}}
\newcommand{\afb}{\mathfrak{a}_{\bot}}
\newcommand{\hf}{\mathfrak{h}}
\newcommand{\hfb}{\mathfrak{h}_{\bot}}

\begin{document}

\title{Recursive algorithm and branching for nonmaximal embeddings}
\author{V D Lyakhovsky$^1$ and A A Nazarov$^2$}
\address{$^{1,2}$ Theoretical Department, SPb State University,
198904, Sankt-Petersburg, Russia }
\eads{$^1$ \mailto{lyakh1507@nm.ru}, $^2$ \mailto{antonnaz@gmail.com}}

\begin{abstract}
  Recurrent relations for branching coefficients in affine Lie algebras
  integrable highest weight modules are studied. The decomposition algorithm
  based on the injection fan technique is developed for the case of an arbitrary
  reductive subalgebra. In particular we consider the situation where
  the Weyl denominator becomes singular with respect to the subalgebra.
  We demonstrate
  that for any reductive subalgebra it is possible to define the
  injection fan and the analogue of the Weyl numerator -- the tools that describe
  explicitly the recurrent properties of branching coefficients.
  Possible applications of the fan technique in CFT models are considered.
\end{abstract}
\submitto{\JPA}
%\maketitle

\section{Introduction}
\label{sec:introduction}

The branching problem for affine Lie algebras emerges in conformal field theory, for example,
in the construction of modular-invariant partition functions \cite{difrancesco1997cft}.
Recently the problem of conformal embeddings was considered in the paper \cite{coquereaux2008conformal}.

There are different approaches to deal with the branching coefficients. Some of them use the BGG
resolution \cite{bernstein1975differential} (for Kac-Moody algebras the algorithm is described in
\cite{kac1990idl},\cite{wakimoto2001idl}), the Schur function series \cite{fauser2006new}, the BRST
cohomology \cite{Hwang:1994yr}, Kac-Peterson formulas \cite{kac1990idl,quella2002branching} or the
combinatorial methods applied in \cite{feigin707principal}.

In this paper we prove that
for an arbitrary reductive subalgebra the branching coefficients are subject to
the recurrent properties that can be explicitly formulated and that there exists
an effective and simple algorithm to solve these recurrent relations step by step.
The basic idea is similar to the one used in \cite{ilyin812pbc} for maximal embeddings.
In our case the algorithm is essentially different, new properties of singular weights
are determined to deal with an arbitrary reductive injection $\frak{a} \rightarrow \frak{g}$.

The principal point is to consider the subalgebra $\af$ together with its
counterpart $\afb$ orthogonal to $\af$.
For any reductive algebra $\af$ the subalgebra $\afb \subset \frak{g} $ is regular and reductive.
For a highest weight module $L^{\left( \mu \right)}$ and orthogonal pair of subalgebras
$\left(  \af, \afb \right)$ we consider
the so called singular element $\Psi^{\left( \mu \right)}$ (the numerator
in the Weyl character formula
$ch\left( L^{\mu }\right) =\frac{\Psi ^{\left( \mu \right) }}{\Psi ^{\left( 0\right) }}$,
see for example \cite{humphreys1997introduction})
the Weyl denominator $\Psi ^{\left( 0\right) }_{\afb}$ and the projection
$\Psi ^{\left( \mu \right) }_{\left(  \af, \afb \right)}
=\pi_{\af}\frac{\Psi ^{\left( \mu \right) }_{\frak{g}}}{\Psi ^{\left( 0\right) }_{\afb}}$.
We prove that for any highest weight $\hf$-diagonalizable module $L^{\left( \mu \right)}$ and orthogonal pair
$\left(  \af, \afb \right)$ the element
$\Psi ^{\left( \mu \right) }_{\left(  \af, \afb \right)}$ has a decomposition with respect to
the set of Weyl numerators $\Psi ^{\left( \mu \right) }_{ \afb }$ of $\afb$.
This decomposition provides the possibility to construct the recurrent property for branching coefficients corresponding
to the injection $\frak{a} \rightarrow \frak{g} $.
The property is formulated in
terms of a specific element $\Gamma_{\af \rightarrow \gf}$ of the group algebra
$\mathcal{E}\left( \frak{g} \right)$ called "the injection fan".
Using this tool we formulate a simple and
explicit algorithm for branching coefficients computations applicable for an arbitrary (maximal or nonmaximal)
subalgebras of finite-dimensional or affine Lie algebras.
In the case of maximal embedding the corresponding fan is unsubtracted, the singular element
becomes trivial
$\Psi ^{\left( \mu \right) }_{\left(  \af, \afb \right)}=\Psi ^{\left( \mu \right) }_{\left(  \gf\right)}$
and the relations described earlier in \cite{ilyin812pbc} are reobtained.

We demonstrate that our algorithm is effective and can be used in studies
of conformal embeddings and coset constructions in rational conformal field theory.

The paper is organized as follows. In the subsection \ref{sec:notation}  we fix the general notations.
In the Section \ref{sec:recurr-form-branch} we derive the decomposition formula based on
recurrent properties of anomalous branching coefficients and describe the decomposition algorithm
for integrable highest weight modules
$L_{\mathfrak{g}}$ with respect to a reductive subalgebra $\mathfrak{a}\subset \mathfrak{g}$
(subsection \ref{sec:algorithm}). In the Section \ref{sec:finite-dimens-lie} we present several
simple examples for finite-dimensional Lie algebras. The affine Lie algebras and their applications in
CFT models are considered in Section \ref{sec:phys-appl}.
The general properties of the proposed algorithm and
possible further developments are discussed (Section \ref{sec:conclusion}).

\subsection{Notation}
\label{sec:notation}

Consider affine Lie algebras $\frak{g}$ and $\af$ with the
underlying finite-dimensional subalgebras $\go$ and $%
\ao$ and an injection $\af\longrightarrow \frak{g%
}$ such that $\af$ is a reductive subalgebra $\frak{a\subset g}$ with
correlated root spaces: $\frak{h}_{\af}^{\ast }\subset \frak{h}_{\frak{g%
}}^{\ast }$ and $\frak{h}_{\ao}^{\ast }\subset \frak{h%
}_{\go}^{\ast }$\
.
We use the following notations:

$L^{\mu }$\ $\left( L_{\af}^{\nu }\right) $\ --- the integrable module
of $\frak{g}$ with the highest weight $\mu $\ ; (resp. integrable $\af$
-module with the highest weight $\nu $ );

$r$ , $\left( r_{\af}\right) $ --- the rank of the algebra $\frak{g}$ $%
\left( \mbox{resp. }\af\right) $ ;

$\Delta $ $\left( \Delta _{\af}\right) $--- the root system; $\Delta
^{+} $ $\left( \mbox{resp. }\Delta _{\af}^{+}\right) $--- the positive
root system (of $\frak{g}$ and $\af$ respectively);

$\mathrm{mult}\left( \alpha \right) $ $\left( \mathrm{mult}_{\af}\left(
\alpha \right) \right) $ --- the multiplicity of the root $\alpha$ in $\Delta
$ (resp. in $\left( \Delta _{\af}\right) $);

$\co{\Delta}$ , $\left( \co{\Delta _{\af}}%
\right)$ --- the finite root system of the subalgebra $\co{%
\frak{g}}$ (resp. $\co{\af}$);

$\mathcal{N}^{\mu }$ , $\left( \mathcal{N}_{\af}^{\nu }\right) $ --- the
weight diagram of $L^{\mu }$ $\left( \mbox{resp. }L_{\af}^{\nu }\right)
$ ;

$W$ , $\left( W_{\af}\right) $--- the corresponding Weyl group;

$C$ , $\left( C_{\af}\right) $--- the fundamental Weyl chamber;

$\bar{C}, \left(\bar{C_{\mathfrak{a}}}\right)$ --- the closure of the fundamental Weyl chamber;

$\rho $\ , $\left( \rho _{\af}\right) $\ --- the Weyl vector;

$\epsilon \left( w\right) :=\det \left( w\right) $ ;

$\alpha _{i}$ , $\left( \beta _{j}\right) $ --- the $i
$-th (resp. $j$-th) basic root for $\frak{g}$ $\left( \mbox{resp. }\af%
\right) $; $i=0,\ldots ,r$,\ \ $\left( j=0,\ldots ,r_{\af}\right) $;

$\delta $ --- the imaginary root of $\frak{g}$ (and of $\af$ if any);

$\alpha _{i}^{\vee }$ , $\left( \alpha _{\left( \af\right) j}^{\vee
}\right) $--- the basic coroot for $\frak{g}$ $\left( \mbox{resp. }\af%
\right) $ , $i=0,\ldots ,r$ ;\ \ $\left( j=0,\ldots ,r_{\af}\right) $;

$\co{\xi }$ , $\co{\xi _{\left( \af\right) }}$
--- the finite (classical) part of the weight $\xi \in P$ , $\left( \mbox{%
resp. }\xi _{\left( \af\right) }\in P_{\af}\right) $;

$\lambda =\left( \co{\lambda };k;n\right) $ --- the
decomposition of an affine weight indicating the finite part $\co{\lambda }$, level $k$ and grade $n$;

$P$ $\left( \mbox{resp. } P_{\af}\right) $ \ --- the weight lattice;

$m_{\xi }^{\left( \mu \right) }$ , $\left( m_{\xi }^{\left( \nu \right)
}\right) $ --- the multiplicity of the weight $\xi \in P$ \ $\left( \mbox{%
resp. }\in P_{\af}\right) $ in the module $L^{\mu }$ , (resp. $\xi \in
L_{\af}^{\nu } $);

$ch\left( L^{\mu }\right) $ $\left( \mbox{resp. }ch\left( L_{\af}^{\nu
}\right) \right) $--- the formal character of $L^{\mu }$ $\left( \mbox{resp. }%
L_{\af}^{\nu }\right) $;

$ch\left( L^{\mu }\right) =\frac{\sum_{w\in W}\epsilon (w)e^{w\circ (\mu
+\rho )-\rho }}{\prod_{\alpha \in \Delta ^{+}}\left( 1-e^{-\alpha }\right) ^{%
\mathrm{{mult}\left( \alpha \right) }}}$ --- the Weyl-Kac formula;

$R:=\prod_{\alpha \in \Delta ^{+}}\left( 1-e^{-\alpha }\right) ^{\mathrm{{%
mult}\left( \alpha \right) }}\quad $
$\left( \mbox{resp. }R_{\af}:=\prod_{\alpha \in \Delta _{%
\af}^{+}}\left( 1-e^{-\alpha }\right) ^{\mathrm{mult}_{\af}\mathrm{%
\left( \alpha \right) }}\right) $--- the Weyl denominator.

\section{Recurrent relations for branching coefficients.}
\label{sec:recurr-form-branch}

Consider the integrable module $L^{\mu }$
of $\frak{g}$ with the highest weight $\mu $ and
let $\af\subset \frak{g}$ be a reductive subalgebra of $\frak{g}$.
With respect to $\af$ the module $L^{\mu }$ is completely reducible,
\begin{equation*}
 L_{\frak{g}\downarrow \af}^{\mu }=\bigoplus
\limits_{\nu \in P_{\af}^{+}}b_{\nu }^{\left( \mu \right) }L_{\af}^{\nu }.
\end{equation*}
Using the projection operator $\pi_{\af}$ (to the weight space $\frak{h_a}^*$)
one can  rewrite this decomposition in terms of formal characters:
\begin{equation}
\label{branching1}
 \pi _{\af}\circ ch\left( L^{\mu }\right)
 =\sum_{\nu \in P_{\af}^{+}}b_{\nu }^{(\mu)}ch\left( L_{\af}^{\nu }\right) .
\end{equation}
We are interested in branching coefficients $b^{(\mu)}_{\nu}$.

\subsection{Orthogonal subalgebra and injection fan.}
\label{subsec:branching-orthog-pair}

In this subsection we shall introduce some simple constructions that will be used
in our studies of branching and in particular the "orthogonal partner" $\afb$ for a
reductive subalgebra $\af$  in  $\gf$.

In the Weyl-Kac formula both numerator and denominator  can be considered
as formal elements containing the singular weights of the Verma modules $V^{\xi}$
with the highest weights $\xi=\mu$ and $\xi=0$ \cite{humphreys1997introduction}.
We attribute singular elements to the corresponding integrable modules $L^{\mu }$
and $L_{\af}^{\nu }$:
\begin{equation*}
\Psi ^{\left( \mu \right) }:=\sum\limits_{w\in W}\epsilon (w)e^{w\circ (\mu +\rho )-\rho },
\end{equation*}
\begin{equation*}
\Psi _{ \af}^{\left( \nu \right) }:=
\sum\limits_{w\in W_{\af}}\epsilon (w)e^{w\circ (\nu +\rho
_{_{\af}})-\rho _{_{\af}}}.
\end{equation*}
and use the Weyl-Kac formula in the form
\begin{equation}
\label{Weyl-Kac2}
ch\left( L^{\mu }\right) =\frac{\Psi ^{\left( \mu \right) }}
{\Psi ^{\left( 0 \right) }}=\frac{\Psi ^{\left( \mu \right) }}{R}.
\end{equation}

Applying formula (\ref{Weyl-Kac2}) to the branching rule (\ref{branching1})
we get the relation connecting the
singular elements $\Psi ^{\left( \mu \right) }$ and $\Psi _{ \af}^{\left( \nu \right) }$ :
\begin{eqnarray}
\nonumber
\pi _{\af}\left( \frac{\sum_{w \in W}\epsilon (w )e^{w
(\mu +\rho )-\rho }}{\prod_{\alpha \in \Delta ^{+}}(1-e^{-\alpha })^{\mathrm{%
mult}(\alpha )}}\right) &=&\sum_{\nu \in P_{\af}^{+}}b_{\nu }^{(\mu )}%
\frac{\sum_{w \in W_{\af}}\epsilon (w )e^{w (\nu +\rho _{%
\af})-\rho _{\af}}}{\prod_{\beta \in \Delta _{\af%
}^{+}}(1-e^{-\beta })^{\mathrm{mult}_{\af}(\beta )}},  \label{eq:4} \\
\pi _{\af}\left( \frac{\Psi ^{\left( \mu \right) }}{R}\right)
&=&\sum_{\nu \in P_{\af}^{+}}b_{\nu }^{(\mu )}\frac{\Psi _{ \frak{%
a}}^{\left( \nu \right) }}{R_{\af}}.
\end{eqnarray}
Here $\Delta _{\af}^{+}$ is the set of
positive roots of the subalgebra $\af$ (without loss of generality we consider
them as vectors from the positive root space $\frak{h}^{\ast  +}$ of $\frak{g}$).

Consider the root subspace
$\frak{h}_{\perp \af}^{\ast }$ orthogonal to  $\af$,
\begin{equation*}
\frak{h}_{\perp \af}^{\ast }:=\left\{ \eta \in \frak{h}^{\ast }
|\forall h \in \hf_{\af};  \eta\left(h \right)=0 \right\} ,
\end{equation*}
and the roots (correspondingly -- positive roots) of $\frak{g}$ orthogonal
to $\af$,
\begin{eqnarray*}
\Delta _{\af_{\perp }} &:&=\left\{ \beta \in \Delta _{\frak{g}}|
\forall h \in \hf_{\af};  \beta\left(h \right)=0  \right\} , \\
\Delta _{\af_{\perp }}^{+} &:&=\left\{ \beta ^{+}\in \Delta _{\frak{g}%
}^{+}|\forall h \in \hf_{\af};  \beta^{+}\left(h \right)=0  \right\} .
\end{eqnarray*}
Let $W_{\af_{\perp }}$ be the subgroup of $W$ generated by the
reflections $w _{\beta }$ for the roots $\beta \in \Delta _{\af%
_{\perp }}^{+}$ . The subsystem $\Delta _{\af_{\perp }}$ determines the
subalgebra $\af_{\perp }$ with the Cartan subalgebra $\frak{h}_{\af%
_{\perp }}$. Let
\begin{equation*}
\frak{h}_{\perp }^{\ast }:=\left\{ \eta \in \frak{h}_{\perp \af}^{\ast
}|\forall h \in \hf_{\af\oplus \af_{\perp}}; \eta \left( h \right)=0 \right\}
\end{equation*}
and consider the subalgebras
\begin{eqnarray*}
\widetilde{\af_{\perp }} &:&=\af_{\perp }\oplus \frak{h}_{\perp }
\\
\widetilde{\af} &:&=\af\oplus \frak{h}_{\perp }.
\end{eqnarray*}
Algebras $\af$ and $\af_{\perp }$ form the ''orthogonal pair''
$\left( \af,\af_{\perp}\right) $
of subalgebras in $\frak{g}$.

For the Cartan subalgebras we have the decomposition
\begin{equation}
\frak{h}=\frak{\frak{h}_{\af}}\oplus \frak{h}_{\af_{\perp }}\oplus
\frak{h}_{\perp }=\frak{\frak{h}_{\widetilde{\af}}}\oplus \frak{h}_{%
\af_{\perp }}=\frak{\frak{h}_{\widetilde{\af_{\perp }}}}\oplus
\frak{h}_{\af}.
\end{equation}
For the subalgebras of an orthogonal pair $\left( \af,\af_{\perp
}\right) $ we consider the corresponding Weyl vectors, $\rho _{\af}$
and $\rho _{\af_{\perp }}$ , and\ form the so called ''defects'' $%
\mathcal{D}_{\af}$ and $\mathcal{D}_{\af_{\perp }}$ of the
injection:
\begin{equation}
\mathcal{D}_{\af}:=\rho _{\af}-\pi _{\af}\rho ,
\end{equation}
\begin{equation}
\label{defect-perp}
\mathcal{D}_{\af_{\perp }}:=\rho _{\af_{\perp }}-\pi _{\af%
_{\perp }}\circ\rho .
\end{equation}
For the highest weight module $L_{\frak{g}}^{\mu }$ consider the singular
weights $\left\{\left( w(\mu +\rho )-\rho \right)|w  \in W \right\}$ and
their projections to $h_{\widetilde{\af_{\perp }}}^{\ast }$ (additionally
shifted by the defect $-\mathcal{D}_{\af_{\perp }}$):
\begin{equation*}
\mu _{\widetilde{\af_{\perp }}}\left( w\right) :=\pi _{\widetilde{\frak{%
a}_{\perp }}}\circ\left[ w(\mu +\rho )-\rho \right] -\mathcal{D}_{\af_{\perp
}},\quad w\in W.
\end{equation*}
Among the weights $\left\{\mu _{\widetilde{\af_{\perp }}}\left( w\right)
|w\in W\right\}$ choose those located in the fundamental chamber $\overline{C_{%
\widetilde{\af_{\perp }}}}$ and let $U$ be the set of representatives $%
u $ for the classes $W/W_{\af_{\perp }}$ such that

\begin{equation}
U:=\left\{ u\in W|\quad \mu _{\widetilde{\af_{\perp }}}\left( u\right)
\in \overline{C_{\widetilde{\af_{\perp }}}}\right\} \quad .
\label{U-def}
\end{equation}
For the same set $U$ introduce the weights
\begin{equation*}
\mu _{\af}\left( u\right) :=\pi _{\af}\circ\left[ u(\mu +\rho )-\rho %
\right] +\mathcal{D}_{\af_{\perp }}.
\end{equation*}
To simplify the form of relations we shall now on omit the sign "$\circ$" in projected
weights.

To describe the recurrent properties for branching coefficients $b_{\nu
}^{(\mu )}$ we shall use the technique elaborated in \cite{ilyin812pbc}. One of the
main tools is the set of weights $\Gamma _{\af\rightarrow \frak{g}%
} $ called the injection fan. As far as we consider the general situation
(where the injection is not necessarily maximal) the notion of the injection fan is
modified:

\begin{definition}
\label{fan-definition} For the product
\begin{equation}
\prod_{\alpha \in \Delta ^{+}\setminus \Delta _{\afb }^{+}}\left( 1-e^{-\pi
_{\af}\alpha }\right) ^{\mathrm{mult}(\alpha )-\mathrm{mult}_{\af%
}(\pi _{\af}\alpha )}=-\sum_{\gamma \in P_{\af}}s(\gamma
)e^{-\gamma }  \label{eq:6}
\end{equation}
consider the carrier $\Phi _{\af\subset \frak{g}}\subset P_{\af}$
of the function $s(\gamma )=\det \left( \gamma \right) $ :
\begin{equation}
\Phi _{\af\subset \frak{g}}=\left\{ \gamma \in P_{\af}|s(\gamma
)\neq 0\right\}   \label{eq:37}
\end{equation}
The ordering of roots in $\co{\Delta _{\af}}$ induce the
natural ordering of the weights in $P_{\af}$. Denote by $\gamma _{0}$
the lowest vector of $\Phi _{\af\subset \frak{g}}$ . The set
\begin{equation}
\Gamma _{\af\rightarrow \frak{g}}=\left\{ \xi -\gamma _{0}|\xi \in \Phi _{%
\af\subset \frak{g}}\right\} \setminus \left\{ 0\right\}
\label{fan-defined}
\end{equation}
is called the \textit{injection fan}.
\end{definition}
In the next subsection we shall see how the injection fan defines the recurrent
properties of branching coefficients. It must be noticed that the injection fan is
the universal instrument that depends only on the injection and doesn't depend on 
the properties of a module.

\subsection{Decomposing the singular element.}
\label{subsec:decomp-sing-element}

Now we shall prove that the Weyl-Kac character formula (in terms of singular
elements) describes the particular case of a more general relation:

\begin{lemma}
\label{lemma}
Let $\left( \af,\afb \right)$ be the orthogonal pair of reductive
subalgebras in $\frak{g}$, with $\widetilde{\af_{\perp }}=\af%
_{\perp }\oplus \frak{h}_{\perp }$ and $\widetilde{\af}=\af\oplus
\frak{h}_{\perp }$ ,

$L^{\mu }$ be the highest weight module with the singular element
$\Psi ^{\left(\mu \right)}$ ,

$R_{\af_{\perp }}$ be the Weyl denominator for $\af_{\perp }$.

Then the element $\Psi ^{\left( \mu \right) }_{\left(  \af, \afb \right)}
=\pi _{\af}\left( \frac{\Psi _{\frak{g}}^{\mu }}{R_{\af_{\perp }}}\right) $
can be decomposed into the sum over $u\in U$ (see (\ref{U-def})) of
the singular weights $e^{\mu _{\af}\left( u\right) }$ with the
coefficients $\epsilon (u)\mathrm{\dim }\left( L_{\widetilde{\af_{\perp
}}}^{\mu _{\widetilde{\af_{\perp }}}\left( u\right) }\right) $:
\begin{equation}
\Psi ^{\left( \mu \right) }_{\left(  \af, \afb \right)}=\quad \pi _{\af}\left( \frac{\Psi^{\mu }}{R_{\af%
_{\perp }}}\right) =\sum_{u\in U}\;\epsilon (u)\mathrm{\dim }
\left( L_{\widetilde{\af_{\perp }}}^{\mu _{%
\widetilde{\af_{\perp }}}\left( u\right) }\right) e^{\mu _{\af}\left( u \right) }.
\end{equation}
\end{lemma}

\begin{proof}
With $u\in U $   and $v\in W_{\afb}$ perform the decomposition
\begin{equation*}
u(\mu +\rho )=\pi _{\left( \af\right) } u(\mu +\rho )+\pi _{\left(
\widetilde{\af_{\perp }}\right) } u(\mu +\rho )
\end{equation*}
for the singular weight $vu(\mu +\rho )-\rho$:
\begin{equation}
\label{sing-decomp-1}
\begin{array}{lcl}
vu(\mu +\rho )-\rho &=&\pi _{\left( \af\right) }\left( u(\mu +\rho
)\right) -\rho +\rho _{\af_{\perp }}+\pi _{\left( \hfb\right) }\rho \\
&& + \ v\left( \pi _{\left( \widetilde{%
\af_{\perp }}\right) }u(\mu +\rho )-\rho _{\af_{\perp }}+\rho _{%
\af_{\perp }}\right) -\rho _{\af_{\perp }} -\pi _{\left( \hfb\right) }\rho.
\end{array}
\end{equation}
Use the defect $\mathcal{D}_{\afb}$ (\ref{defect-perp}) to simplify
the first summand in (\ref{sing-decomp-1}):
\begin{equation*}
\begin{array}{r}
\pi _{\left( \af\right) }\left( u(\mu +\rho )\right) -\rho +\rho _{%
\mathfrak{a}_{\perp }}+\pi _{\left( \hfb\right) }\rho = \\
\pi _{\left( \af\right) }\left( u(\mu +\rho )\right) -\pi _{\af}\rho
-\pi _{\afb}\rho +\rho _{\afb}= \\
=\pi _{\left( \af\right) }\left( u(\mu +\rho )-\rho \right) +%
\mathcal{D}_{\afb},
\end{array}
\end{equation*}
and the second one:
\begin{equation*}
\begin{array}{c}
v\left( \pi _{\left( \widetilde{%
\af_{\perp }}\right) }u(\mu +\rho )-\rho _{\af_{\perp }}+\rho _{%
\af_{\perp }}\right) -\rho _{\af_{\perp }}-\pi _{\left( \hfb\right) }\rho=\\
v\left( \pi _{\left( \widetilde{%
\afb}\right) }u(\mu +\rho )
- \mathcal{D}_{\afb} - \pi _{\left( \afb\right) }\rho-\pi _{\left( \hfb\right) }\rho
+\rho _{\afb}\right) -\rho _{\afb}=\\
=v\left( \pi _{\left( \widetilde{%
\afb}\right) }\left[ u(\mu +\rho )-\rho\right]
- \mathcal{D}_{\afb}
+\rho _{\afb}\right) -\rho _{\afb}.
\end{array}
\end{equation*}
These expressions provide a kind of a factorization in the anomalous element $\Psi^{\mu
}$ and we find in it the combination of
anomalous elements $\Psi _{\widetilde{\af_{\perp }}}^{\eta }$ of the
subalgebra $\widetilde{\af_{\perp }}$-modules $L_{\widetilde{\af%
_{\perp }}}^{\eta }$:
\begin{equation*}
\begin{array}{l}
\Psi^{\mu }=\sum_{u\in U}\sum_{v\in W_{\af_{\perp }}}
\epsilon (v)\epsilon (u)e^{vu(\mu +\rho )-\rho }= \\
=\sum_{u\in U}\epsilon (u)e^{\pi _{\af}\left[ u(\mu +\rho )-\rho \right]
+\mathcal{D}_{\af_{\perp }}}\sum_{v\in W_{\af_{\perp }}}\epsilon
(v)e^{v\left( \pi _{\left( \widetilde{\af_{\perp }}\right) }\left[
u(\mu +\rho )-\rho \right] -\mathcal{D}_{\af_{\perp }}+\rho _{\af%
_{\perp }}\right) -\rho _{\af_{\perp }}}= \\
=\sum_{u\in U}\;\epsilon (u)e^{\pi _{\left( \af\right) }\left[ u(\mu
+\rho )-\rho \right] +\mathcal{D}_{\af_{\perp }}}\Psi _{\widetilde{%
\af_{\perp }}}^{\pi _{\left( \widetilde{\af_{\perp }}\right) }%
\left[ u(\mu +\rho )-\rho \right] -\mathcal{D}_{\af_{\perp }}}
\end{array}
\end{equation*}

Dividing both sides by the Weyl element $R_{\af_{\perp }}=\prod_{\beta
\in \Delta _{\af_{\perp }}}(1-e^{-\beta })^{\mathrm{mult}(\beta )}$ and
projecting them to the weight space $h_{\af}^{\ast }$\ we obtain the
desired relation:
\begin{eqnarray*}
\Psi ^{\left( \mu \right) }_{\left(  \af, \afb \right)}
&=&\sum_{u\in W/W_{\af_{\perp }}}\;\epsilon (u)e^{\pi _{\frak{a%
}}\left[ u(\mu +\rho )-\rho \right] }\pi _{%
\af}\left( \frac{\Psi _{\widetilde{\af_{\perp }}}^{\pi _{\left(
\widetilde{\af_{\perp }}\right) }\left[ u(\mu +\rho )-\rho \right] -%
\mathcal{D}_{\af_{\perp }}}}{\prod_{\beta \in \Delta _{\af_{\perp
}}}(1-e^{-\beta })^{\mathrm{mult}(\beta )}}\right)  \\
&=&\sum_{u\in U}\;\epsilon (u)\mathrm{\dim }\left( L_{\widetilde{\af%
_{\perp }}}^{\mu _{\widetilde{\af_{\perp }}}\left( u\right) }\right)
e^{\pi _{\af}\left[ u(\mu +\rho )-\rho \right] }.
\end{eqnarray*}
\end{proof}

\begin{remark}
This relation can be considered a generalized form of the Weyl formula for singular
element $\Psi _{\frak{g}}^{\mu }$ : the vectors $\mu _{\af}\left(
u\right) $ play the role of singular weights while instead of the determinants $%
\epsilon (u)$ we have the products $\epsilon (u)\mathrm{\dim }\left( L_{%
\widetilde{\af_{\perp }}}^{\mu _{\widetilde{\af_{\perp }}}\left(
u\right) }\right) .$ In fact when $\frak{a=g}$ both $\af_{\perp }$ and $%
\frak{h}_{\perp }$ are trivial, $U=W$ , and\ the original Weyl formula is
easily reobtained.
\end{remark}

\subsection{Constructing recurrent relations.}
\label{subsec:Construct-recurrent-rel}

Consider the right-hand side of relation (\ref{eq:4}).
The numerator there describes the branching in terms of singular elements and
it is reasonable to expand it as an element of $\mathcal{E}\left( \frak{g} \right)$:
\begin{equation}
  \label{eq:21}
  \sum_{\nu \in \bar{C_{\mathfrak{a}}}}b_{\nu }^{\left( \mu \right) }\Psi _{\left( \frak{%
        a}\right) }^{\left( \nu \right) }=\sum_{\lambda \in P_{\af}}k_{\lambda
  }^{\left( \mu \right) }e^{\lambda }.
\end{equation}
Here the coefficients $k_{\lambda}^{\left( \mu \right) }$ are integer and their signs
depend on the length (see \cite{humphreys1997introduction})  of the Weyl group elements in
$\Psi _{\left( \frak{a}\right) }^{\left( \nu \right) }$. The important property of
$k_{\lambda}^{\left( \mu \right) }$'s is that they coincide with the branching coefficients
for all weights $\nu$ inside the main Weil chamber:
\begin{equation}
%  \label{eq:20}
  b^{(\mu)}_{\nu}=k^{(\mu)}_{\nu} \; \mbox{for} \; \nu\in \bar{C}_{\mathfrak{a}}.
\label{eq:21-1}
\end{equation}
We call the coefficients $k_{\lambda}$ --- the anomalous branching coefficients
(see also \cite{ilyin812pbc}).

Now we can state the main theorem which gives us an instrument for the
recurrent computation of branching coefficients.

\begin{theorem}
  For the anomalous branching coefficients $k^{(\mu)}_{\nu}$ (\ref{eq:21})
  the following relation holds
  \begin{equation}
    \label{recurrent-relation}
    \begin{array}{c}
      k_{\xi }^{\left( \mu \right) }=-\frac{1}{s\left( \gamma _{0}\right) }\left(
        \sum_{u\in U} \epsilon(u)\;
        \dim \left( L_{\widetilde{\af_{\perp }}}^{\mu
        _{\widetilde{\af_{\perp }}}\left( u\right) }\right)
        \delta_{\xi-\gamma_0,\pi_{\af}(u(\mu+\rho)-\rho)}+ \right.\\
      \left.
        +\sum_{\gamma \in
          \Gamma _{\af \rightarrow \gf}}s\left( \gamma +\gamma _{0}\right) k_{\xi
          +\gamma }^{\left( \mu \right) }\right).
    \end{array}
  \end{equation}
\end{theorem}
\begin{proof}
Redress the relation (\ref{eq:4}) for the element $
\frac{\Psi _{\frak{g}}^{\mu }}{R_{\af_{\perp }}}$ using definition (\ref{eq:37})
for the carrier $\Phi _{\af\subset \frak{g}}$ ,
\begin{equation*}
\begin{array}{l}
\Psi ^{\left( \mu \right) }_{\left(  \af, \afb \right)}
=\pi _{\frak{a}}\left( \frac{\Psi _{\frak{g}}^{\mu }}{R_{\af_{\perp }}}%
\right) = \\[2mm]
=\prod\limits_{\alpha \in \Delta ^{+}\setminus \Delta _{\afb }^{+}}\left(
1-e^{-\pi _{\af}\alpha }\right) ^{\mathrm{mult}(\alpha )-\mathrm{mult}_{%
\af}(\pi _{\af}\alpha )}\left( \sum\limits_{\nu \in P_{\af%
}^{+}}b_{\nu }^{(\mu )}\sum\limits_{w\in W_{\af}}\epsilon (w)e^{w(\nu
+\rho _{\af})-\rho _{\af}}\right) = \\[5mm]
=-\sum\limits_{\gamma \in \Phi _{\af\subset \frak{g}}}s(\gamma
)e^{-\gamma }\left( \sum\limits_{\nu \in P_{\af}^{+},w\in W_{\af%
}}\epsilon (w)b_{\nu }^{(\mu )}e^{w(\nu +\rho _{\af})-\rho _{\af%
}}\right)  \\
=-\sum\limits_{\gamma \in \Phi _{\af\subset \frak{g}}}s(\gamma
)e^{-\gamma }\left( \sum\limits_{\nu \in P_{\af}^{+},w\in W_{\af%
}}\epsilon (w)b_{\nu }^{(\mu )}e^{w(\nu +\rho _{\af})-\rho _{\af%
}}\right) .
\end{array}
\end{equation*}

Then expand the sum in brackets (with respect to the formal basis in $\mathcal{E}$%
):
\begin{equation*}
\Psi ^{\left( \mu \right) }_{\left(  \af, \afb \right)}
=-\sum_{\gamma \in \Phi _{\af\subset \frak{g}}}s(\gamma
)e^{-\gamma }\sum_{\lambda \in P_{\af}}k_{\nu }^{(\mu )}e^{\lambda
}=-\sum_{\gamma \in \Phi _{\af\subset \frak{g}}}\sum_{\lambda \in P_{%
\af}}s(\gamma )k_{\nu }^{(\mu )}e^{\lambda -\gamma }.
\end{equation*}
Substitute in the left-hand side the expression obtained in Lemma \ref{lemma},
\begin{eqnarray*}
\Psi ^{\left( \mu \right) }_{\left(  \af, \afb \right)}
&=&\sum_{u\in U}\;\epsilon (u)e^{\pi _{\af}\left( \mu _{\frak{a%
}}\left( u\right) \right) }\dim \left( L_{\widetilde{\af_{\perp }}%
}^{\mu _{\widetilde{\af_{\perp }}}\left( u\right) }\right)
\label{anom modules 2} \\
&=&\sum_{u\in U}\;\epsilon (u)e^{\pi _{\af}\left[ u(\mu +\rho )-\rho %
\right] }\dim \left( L_{\widetilde{\af_{\perp }}}^{\mu _{\widetilde{%
\af_{\perp }}}\left( u\right) }\right)  \\
&=&-\sum_{\gamma \in \Phi _{\af\subset \frak{g}}}\sum_{\lambda \in P_{%
\af}}s(\gamma )k_{\nu }^{(\mu )}e^{\lambda -\gamma }.
\end{eqnarray*}
The immediate consequence of this equality is:
\begin{equation}
\sum_{u\in U}\epsilon (u)\dim \left( L_{\widetilde{\af_{\perp }}}^{\mu
_{\widetilde{\af_{\perp }}}\left( u\right) }\right) \delta _{\xi ,\pi _{%
\af}\left[ u(\mu +\rho )-\rho \right] }+\sum_{\gamma \in \Phi _{\af%
\subset \frak{g}}}s(\gamma )\;k_{\xi +\gamma }^{(\mu )}=0,\quad \xi \in P_{%
\af}.  \label{eq:17}
\end{equation}
The obtained formula means that the coefficients $k_{\xi +\gamma }^{(\mu )}$
for $\gamma \in \Phi _{\af\subset \frak{g}}$ are not independent, they
are subject to the linear relations and the form of these relations changes
when the tested weight $\xi $ coincides with one of \ the ''singular
weights''  $\left\{ \pi _{\af}\left[ u(\mu +\rho )-\rho \right] |u\in
U\right\} $ . To conclude the proof we extract the lowest weight $\gamma _{0}\in \Phi _{\frak{a%
}\subset \frak{g}}$ and pass to the summation over the vectors of the
injection fan $\Gamma _{\af\rightarrow \frak{g}}$ (see the definition \ref
{fan-definition}). Thus we get the desired recurrent relation (\ref{recurrent-relation}).
\end{proof}

\subsection{Embeddings and orthogonal pairs in simple Lie algebras}
\label{sect-embeddings}

In this subsection we discuss some properties of ''orthogonal pairs'' of
subalgebras in simple Lie algebras of classical series.

When both $\frak{g}$ and $\af$ are finite-dimensional all the regular
embeddings can be obtained by a successive elimination of nodes in the
extended Dynkin diagram of $\frak{g}$ (and $\Delta _{\bot }^{+}=\emptyset $
if $\af$ is maximal). For the classical series $A$, $C$ and $D$ when
the regular injection $\af\rightarrow \frak{g}$ is thus fixed, the
Dynkin diagram for $\af_{\bot }$ is obtained from the extended diagram
of $\frak{g}$ by eliminating the subdiagram of $\af$ and the adjacent
nodes:
\begin{table}[tbh]
\label{tab:diagrams} \noindent \centering{\
\begin{tabular}{|l|l|l|}
\hline
$\frak{g}$ & Extended diagram of $\frak{g}$ & Diagrams of the subalgebras $%
\af,\; \afb$   \\ \hline
$A_n$ & \includegraphics{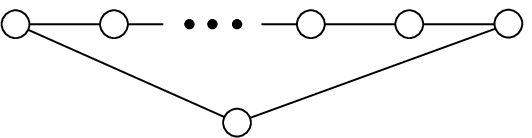} & \includegraphics{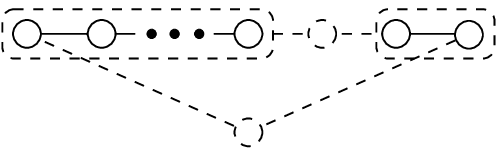}
  \\ \hline
$C_n$ & \includegraphics{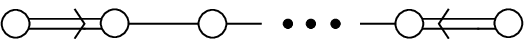} & \includegraphics{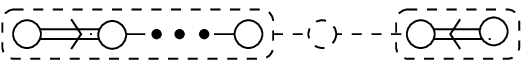}
  \\ \hline
$D_n$ & \includegraphics{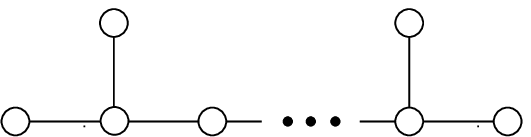} & \includegraphics{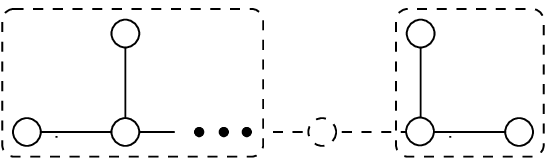}
  \\ \hline
\end{tabular}
}
\caption{Subalgebras $\af,\;\af_{\bot }$ for the classical series}
\end{table}

In the case of $B$ series the situation is different. The reason is that
here the subalgebra $\af_{\bot }$ may be larger than the one obtained
by elimination of the subdiagram of $\af$ and the adjacent nodes. The
subalgebras of the orthogonal pair,$\ \af$ and $\af_{\bot }$, must
not form a direct sum in $\frak{g}$ . It can be directly
checked that when $\frak{g}=B_{r}$ and $\af=B_{r_{\af}}$ the
orthogonal subalgebra is $\af_{\bot }=B_{r-r_{\af}}$. Consider the
injection $B_{r_{\af}}\rightarrow B_{r},\quad 1<r_{\af}<r$. By
eliminating the simple root $\alpha _{r_{\af}-1}=e_{r_{\af%
}-1}-e_{r_{\af}}$ one splits the extended Dynkin diagram of $B_{r}$
into the disjoint diagrams for $\af=B_{r_{\af}}$ and $D_{r-r_{%
\af}}$. But the system $\Delta _{\af_{\perp }}$ contains not only
the simple roots $\left\{ e_{1}-e_{2},e_{2}-e_{3},\ldots ,e_{r_{\af%
}-2}-e_{r_{\af}-1},e_{1}+e_{2}\right\} $ but also the root $e_{r_{\frak{%
a}}-1}$. Thus $\Delta _{\af_{\perp }}$ forms the subsystem of the type $%
B_{r-r_{\af}}$ and the orthogonal pair for the injection $B_{r_{\af%
}}\rightarrow B_{r}$ is $\left( B_{r_{\af}},B_{r-r_{\af}}\right) $%
. In the next Section the particular case of such orthogonal pair is
presented for the injection $B_{2}\rightarrow B_{4}$ (see Figure \ref
{fig:dynkin}).

The complete classification of regular subalgebras for affine Lie algebras
can be found in the recent paper \cite{1751-8121-41-36-365204}. From the complete
classification of maximal special subalgebras in classical Lie algebras \cite
{dynkin1952semisimple} we can deduce the following list of pairs of
orthogonal subalgebras $\af,\;\af_{\bot }$:
\begin{equation*}
\begin{array}{lll}
su(p)\oplus su(q) & \subset su(pq) &  \\
so(p)\oplus so(q) & \subset so(pq) &  \\
sp(2p)\oplus sp(2q) & \subset so(4pq) &  \\
sp(2p)\oplus so(q) & \subset sp(2pq) &  \\
so(p)\oplus so(q) & \subset so(p+q) & \mathrm{{for}\;p\;{and}\;q\;{odd}.}
\end{array}
\end{equation*}

\subsection{Algorithm for recursive computation of branching coefficients}

\label{sec:algorithm}

The recurrent relation (\ref{recurrent-relation}) allows us to formulate an
algorithm for recursive computation of branching coefficients. In this
algorithm there is no need to construct the module $L^{(\mu)}_{\frak{g}}$ or
any of the modules $L^{(\nu)}_{\af}$.

It contains the following steps:

\begin{enumerate}
\item  Construct the root system $\Delta _{\af}$ for the embedding $%
\af\rightarrow \frak{g}$.

\item  Select the positive roots $\alpha \in \Delta ^{+}$ orthogonal
to  $\af$, i.e. form the set $\Delta _{\afb }^{+}$.

\item  Construct the set $\Gamma _{\af\rightarrow \frak{g}}$. The relation
 (\ref{eq:6}) defines the sign function
 $s(\gamma)$ and the set $\Phi_{\af\subset \frak{g}}$ where the lowest weight
 $\gamma_0$ is to be subtracted to get the fan (\ref{fan-defined}):
 $\Gamma _{\af\rightarrow \frak{g}}=\left\{ \xi -\gamma _{0}|\xi \in \Phi _{%
\af\subset \frak{g}}\right\} \setminus \left\{ 0\right\}$.

\item  Construct the set $\widehat{\Psi ^{(\mu )}}=\left\{ w (\mu +\rho
)-\rho ;\;w \in W\right\} $ of singular weights for the $\frak{g}$%
-module $L^{(\mu )}$.

\item  Select the weights $\left\{ \mu _{\widetilde{\af_{\perp }}%
}\left( w\right) =\pi _{\widetilde{\af_{\perp }}}\left[ w(\mu +\rho
)-\rho \right] -\mathcal{D}_{\af_{\perp }}\in \overline{C_{\widetilde{%
\af_{\perp }}}}\right\} $. Since the set $\Delta _{\bot }^{+}$ is fixed
we can easily check wether the weight $\mu _{\widetilde{\af_{\perp }}%
}\left( w\right) $ belongs to the main Weyl chamber $\overline{C_{\widetilde{%
\af_{\perp }}}}$ (by computing its scalar product with the fundamental
weights of $\afb^{+}$).

\item  For the weights $\mu _{\widetilde{\af_{\perp }}}\left( w\right) $
calculate the dimensions of the corresponding modules $\mathrm{\dim }\left(
L_{\widetilde{\af_{\perp }}}^{\mu _{\widetilde{\af_{\perp }}%
}\left( u\right) }\right) $ using the Weyl dimension formula and construct
the singular element $\Psi ^{\left( \mu \right) }_{\left(  \af, \afb \right)}$.

\item  Calculate the anomalous branching coefficients using the
recurrent relation (\ref{recurrent-relation}) and select among them those
corresponding to the weights in the main Weyl
chamber $\overline{C_{\af}}$.
\end{enumerate}

We can speed up the algorithm by
one-time computation of the representatives of the conjugate classes $%
W/W_{\afb }$.

The next section contains examples illustrating the application of this
algorithm.

\section{Branching for finite dimensional Lie algebras}
\label{sec:finite-dimens-lie}

\subsection{Regular embedding of $A_1$ into $B_2$}
\label{sec:regul-embedd-a_1}

Consider the regular embedding $A_1\to B_2$. Simple roots $\alpha_1, \alpha_2$ of $B_2$ are presented as the dashed vectors in the Figure \ref{fig:B2_A1}. We denote the corresponding Weyl reflections by $w_1, w_2$. The simple root $\beta = \alpha_1+2\alpha_2$ of $A_1$ is indicated as the grey vector.

\begin{figure}[p]
  \noindent\centering{
    \includegraphics[width=80mm]{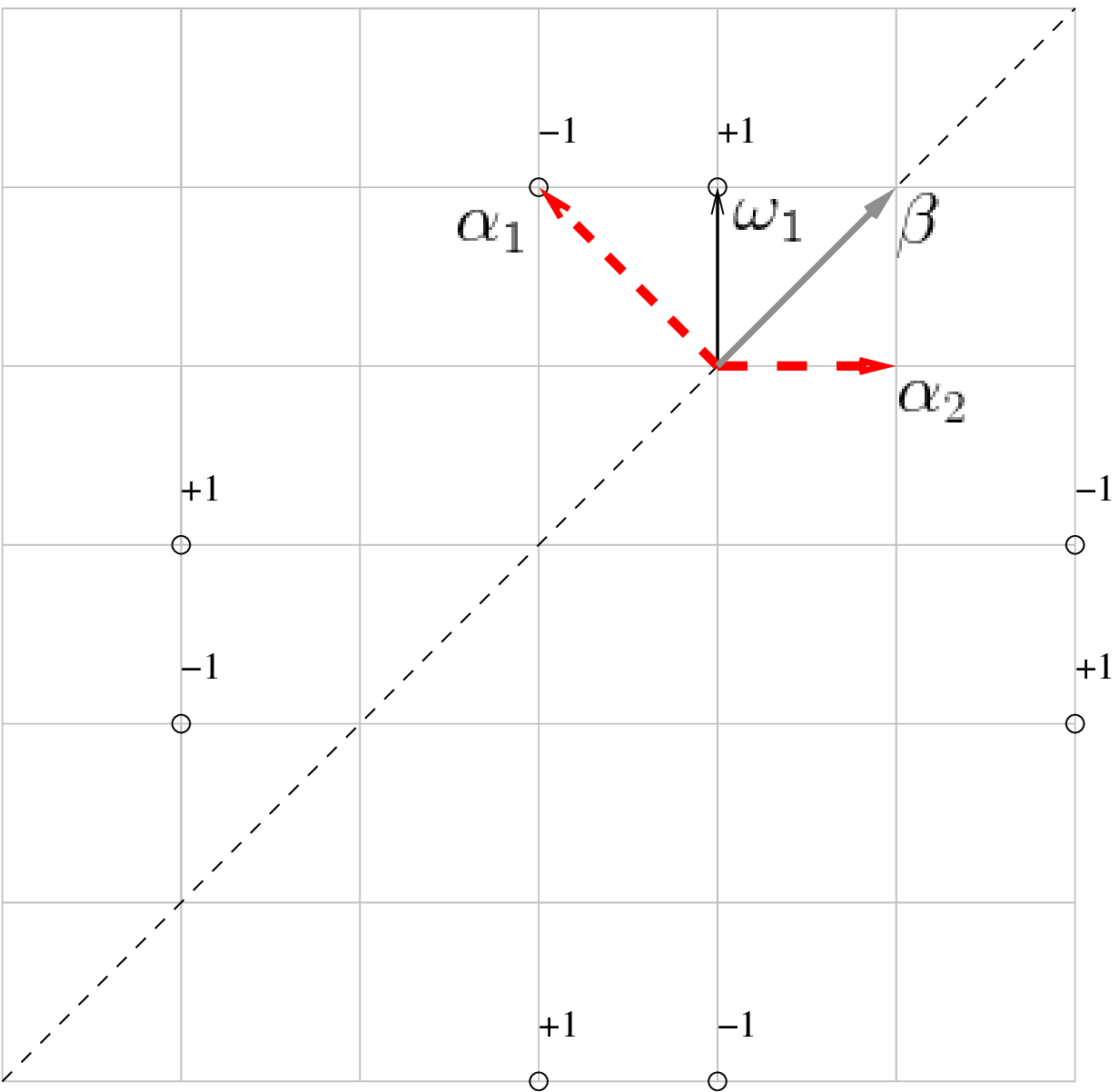}
  }
  \caption{Regular embedding of $A_1$ into $B_2$. Simple roots $\alpha_1, \alpha_2$ of $B_2$ are presented as the dashed vectors.
    The simple root $\beta = \alpha_1+2\alpha_2$ of $A_1$ is indicated as the grey vector.
    The highest weight of the fundamental representation $L^{(1,0)=\omega_1}_{B_2}$ is shown by the black vector.
    The weights of the singular element $\Psi^{(\omega_1)}$ are marked by circles with superscripts indicating
    the corresponding determinants $\epsilon(w)$.}
  \label{fig:B2_A1}

%\end{figure}
%\begin{figure}[pb]
  \noindent\centering{
    \includegraphics[width=80mm]{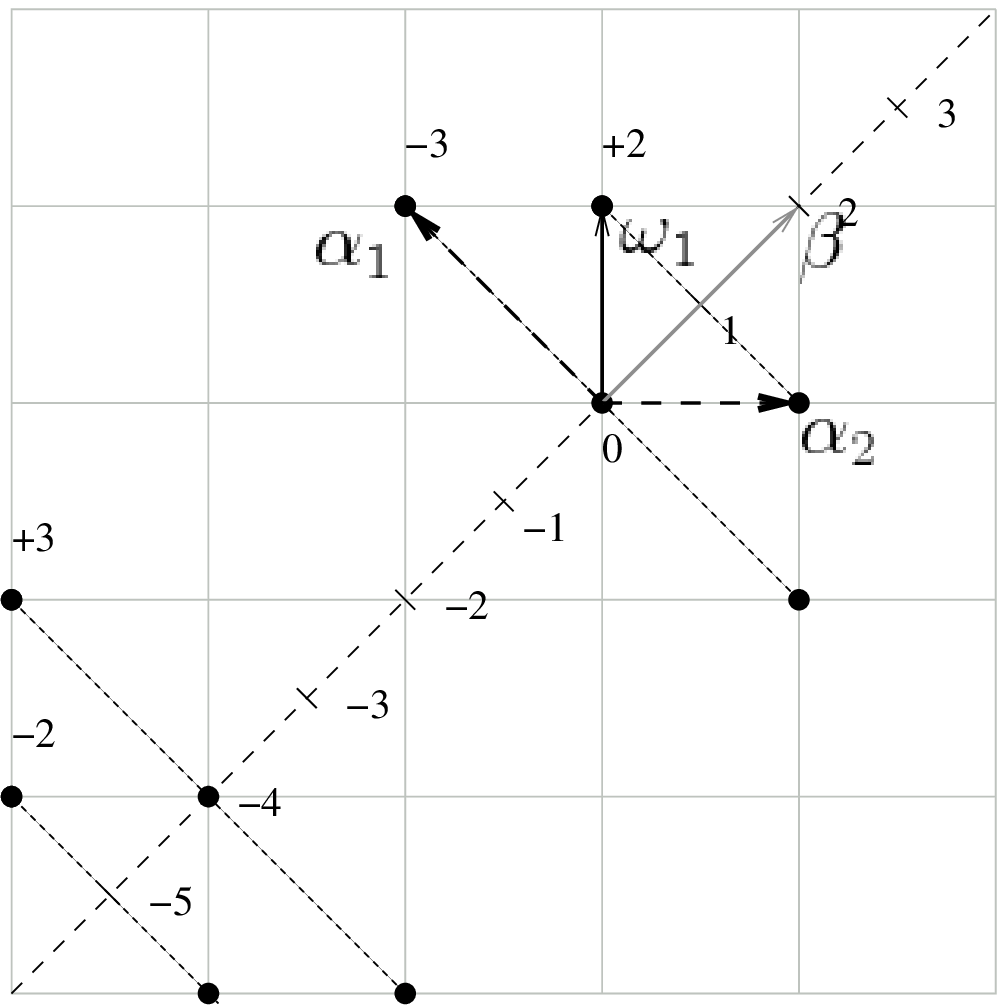}
  }
  \caption{Here in addition to the diagram presented above (Figure(\ref{fig:B2_A1}))
  the weights of $\left( \afb=A_1 \right)$-modules $L_{\af_{\perp }}^{\mu_{\af_{\perp }}\left( u\right) }$
  originating in the points $\pi _{\af}\left[ u(\mu +\rho )-\rho \right] $
are shown by dotted lines.
    The superscripts over the highest weights  $\mu_{\af_{\perp }}\left( u\right)$
    are now the products $\epsilon(u)\dim\left(L_{\af_{\perp }}^{\mu_{\af_{\perp }}\left( u\right) }\right)$.
    Coordinates along the root $\beta$ are counted in terms of the fundamental weight of $\af$. }
\label{fig:B2_A1_2}
\end{figure}

Let's perform the reduction of the fundamental representation $L^{(1,0)=\omega_1}_{B_2}$
($\omega_1$ -- the black vector in Figure \ref{fig:B2_A1}). The root $\alpha_1$ is orthogonal to $\beta$,
so we have $\Delta_{\perp}^+ = \left\{ \alpha_1 \right\}$.
According to the definition (\ref{fan-definition})
the fan $\Gamma_{A_1\to B_2}$ consists of two weights:
\begin{equation*}
  \label{eq:22}
  \Gamma_{A_1\to B_2}=\left\{ (1;2),\; (2;-1) \right\},
\end{equation*}
where the second component is the value of the sign function $s(\gamma)$.
The singular weights  $\left\{ w (\omega_1 +\rho)-\rho ;\;w \in W\right\}$
are indicated by circles with the superscript $\epsilon\left( w \right)$.
The space $U$ is the factor $W/W_{\afb}$ where $W_{\afb}=\left\{e,w_1\right\}$.
This means that the singular weights located above the line generated by $\beta$
belong to the Weyl chamber $\overline{C_{\widetilde{\af_{\perp }}}}$.
According to formula (\ref{defect-perp}) in our case we have $\mathcal{D}_{\af_{\perp }}=0$ and
$\hf_{\perp }=0$, thus $\left\{ \mu _{\af_{\perp }}\left( w\right)
=\pi _{\af_{\perp }}\left[ w(\mu +\rho)-\rho \right]\right\}$.
We obtain four highest weights for $\af_{\perp }$-modules. In terms of
$\af_{\perp }$-fundamental weight $\frac{1}{2} \alpha_1$ these highest weights
$\left\{ \mu _{\af_{\perp }}\left( u\right)
=\pi _{\af_{\perp }}\left[ u(\mu +\rho)-\rho \right]| u \in U \right\}$
 are
$\left\{ \left( 1\right) \left( 2\right) \left( 2\right) \left( 1\right) \right\}$. To
visualize the procedure we indicate explicitly in Figure (\ref{fig:B2_A1_2}) how the
corresponding weight diagrams
$\left\{ \mathcal{N}_{\af_{\perp }}^{\mu _{\af_{\perp }}\left( u\right) }\right\} $
are attached to the set of $\af$-weights
$\left\{ \mu _{\af}\left( u\right)\right\} =\left\{\pi _{\af}\left[ u(\mu +\rho )-\rho \right]\right\}
=\left\{ \left( 1\right) \left( 0\right) \left( -4\right) \left( -5\right) \right\}$.
In fact we do not need the
weight diagrams but only the dimensions of the corresponding modules
$L_{\af_{\perp }}^{\mu_{\af_{\perp }}\left( u\right) }$ multiplied by
$\epsilon \left( u\right) $. The obtained values are to be attributed to the points
$\left\{ \left( 1\right) \left( 0\right) \left( -4\right) \left( -5\right) \right\}$
in $P_{\af}$. The corresponding element of ${\cal E}_{\af}$ is
the singular element $\Psi ^{\left( \mu \right) }_{\left(  \af, \afb \right)}$
with the set of weights having anomalous multiplicities:
\begin{equation}
  \label{eq:25}
  \left\{(1;2),\; (0;-3),\; (-4;3),\; (-5;-2)\right\}.
\end{equation}

Applying formula (\ref{recurrent-relation}) with the fan
$\Gamma_{A_1\to B_2}$ to the set (\ref{eq:25})
we get zeros for the weights
greater than the highest anomalous vector $(1;2)$
and $k^{(1,0)}_1=2$ for the vector $(1;2)$ itself.
For the anomalous weight (0;-3) on the boundary of $\bar{C}^{(0)}_{\af}$ the recurrent relation gives
\begin{equation*}
  \label{eq:23}
  k^{(1,0)}_{0}=-1\cdot k^{(1,0)}_2 +2\cdot k^{(1,0)}_1 - 3\cdot \delta_{0,0} = 1,
\end{equation*}
the branching is completed: $L_{B_2\downarrow A_1}^{\omega_1}=
2L_{A_1}^{\omega_{\left(A_1\right)} }
\bigoplus
L_{A_1}^{2\omega_{\left(A_1\right)} }$.

\subsection{Embedding $B_2$ into $B_4$}
\label{sec:someth-high-dimens}
Consider the regular embedding $B_2 \rightarrow B_4$.
The corresponding Dynkin diagrams are presented in the Figure \ref{fig:dynkin}.
\begin{figure}[h]
  \centering
  \includegraphics[width=50mm]{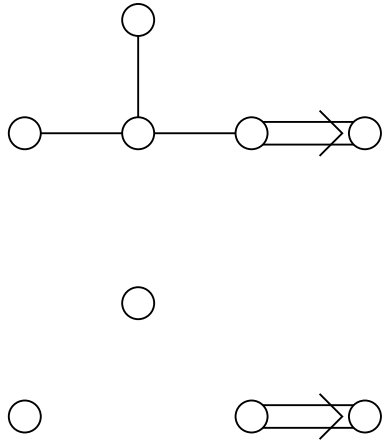}
  \caption{The regular embedding $B_2 \rightarrow B_4$ described by dropping the node from the Dynkin diagram.
  Remember that here $\afb$ is equal to $B_2$ while the diagram
  shows only $A_1\oplus A_1$ (see Subsection \ref{sect-embeddings}).}
  \label{fig:dynkin}
\end{figure}

\begin{figure}[pt]
  \centering
    \includegraphics[width=100mm,height=90mm]{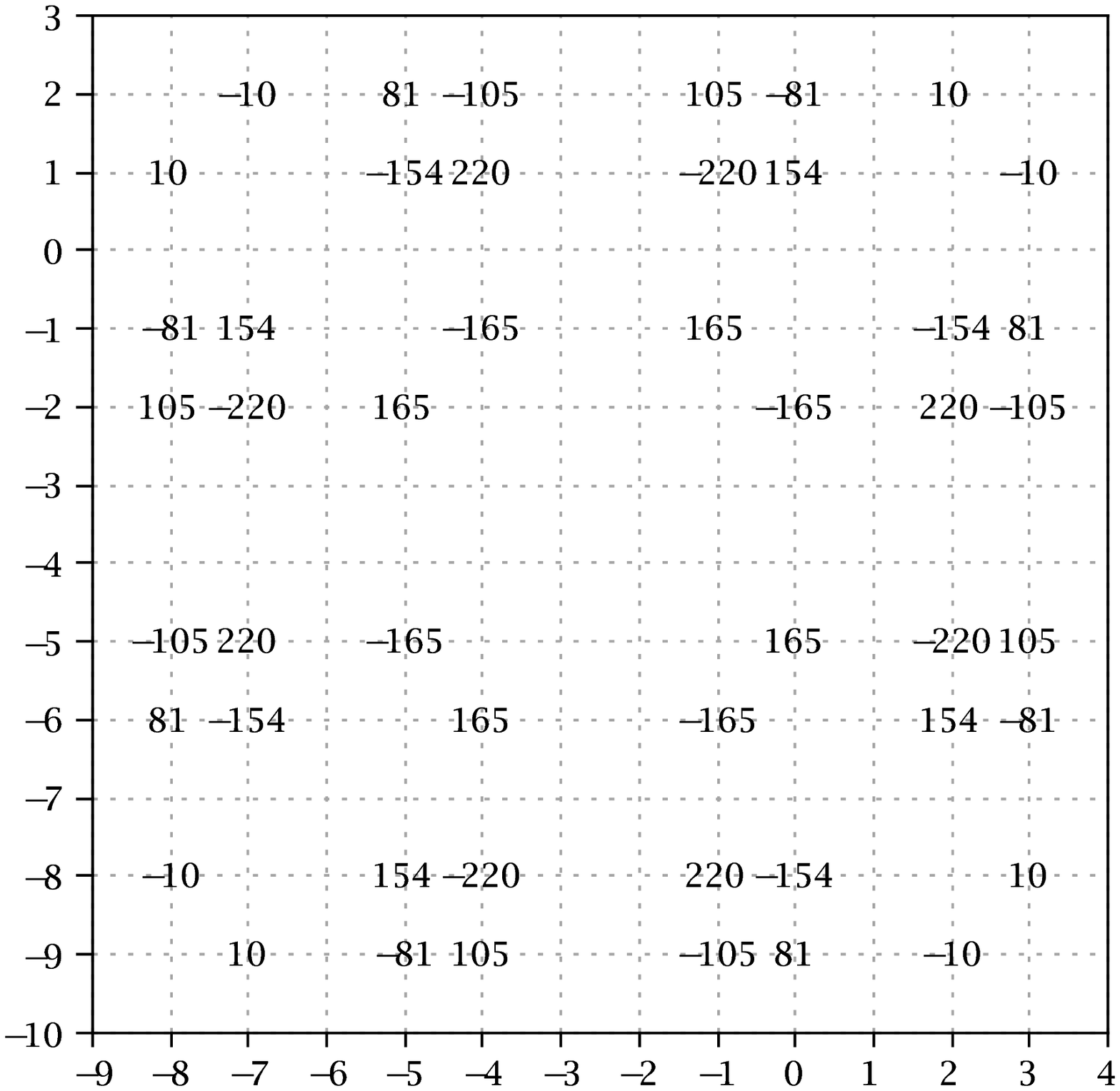}
  \caption{The singular element $e^{\gamma_0}\Psi ^{\left( \mu \right) }_{\left(  \af, \afb \right)}$
  displayed in the weight subspace $P_{\af}$ for $\af=B_2$ with the basis $\left\{e_3,e_4\right\}$.
  We see the projected singular weights $\left\{\pi _{\af}\left[ u(\mu +\rho )-\rho \right] +\gamma_0 | u \in U \right\}$
  shifted by $\gamma_0$ and supplied by the multipliers
  $\epsilon(u)\dim\left(L_{\af_{\perp }}^{\mu_{\af_{\perp }}\left( u\right) }\right)$.}
  \label{fig:B4B2anom}

%\end{figure}
%
%\begin{figure}[pb]
  \centering
  \includegraphics[height=80mm]{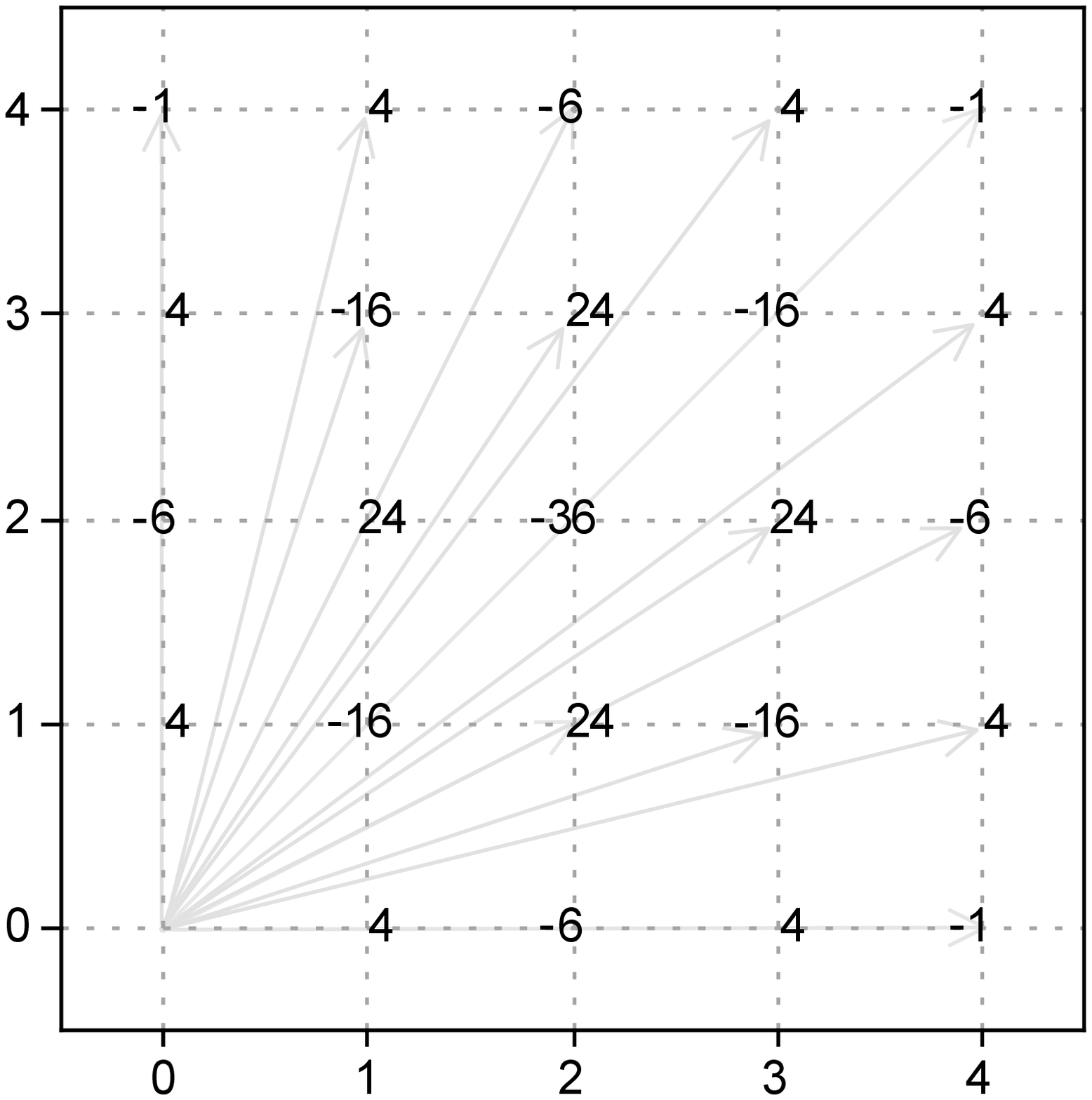}
  \caption{The fan $\Gamma$ for $B_2\rightarrow B_4$ and the values of $s(\gamma+\gamma_0)$ for the weights $\gamma$.}
  \label{fig:B4B2Fan}
\end{figure}

In the orthonormal basis $\left\{e_1,\dots,e_4\right\}$ simple roots and positive roots of $B_4$ are
\begin{eqnarray*}
  \label{eq:19}
  S_{B_4}= \{e_1 - e_2,\; e_2 - e_3,\; e_3 - e_4,\; e_4\},\\[2mm]
 \Delta^+_{B_4}=\left\{ (e_1 - e_2,\; e_2 - e_3,\; e_3 - e_4,\; e_4,\; e_1 - e_3,\; e_2 - e_4,\; e_3 + e_4,\; e_3,\; e_1 - e_4,\;\right.\\
 \left. e_2 + e_4,\; e_2,\; e_1 + e_4,\; e_2 + e_3,\; e_1,\; e_1 + e_3,\; e_1 + e_2\right\}
\end{eqnarray*}
The subalgebra $\af=B_2$ is fixed by the simple roots
\begin{equation*}
  \label{eq:26}
 S_{B_2}=\{e_3-e_4,e_4\}
\end{equation*}
Its orthogonal counterpart $\afb=B_2$ has
\begin{eqnarray*}
  \label{eq:27}
  S_{\afb}=\{e_1-e_2,e_2\},\\
 \Delta^{+}_{\afb}= \left\{e_1-e_2,e_1+e_2,e_1,e_2\right\}.
\end{eqnarray*}
As far as the set $\Delta^+_{B_4} \setminus  \Delta^{+}_{\afb}$ is fixed the injection
fan $\Gamma_{B_2 \to B_4}$ can be constructed using Definition \ref{fan-definition}.
As far as for this injection $s\left( \gamma_0\right)=-1$ in the recursion formula we need just the factor
 $s\left(\gamma + \gamma_0\right)$.
The result is presented in Figure \ref{fig:B4B2Fan}.

Consider the $B_4$-module $L^{\mu}$ with the highest weight $\mu=2e_1 + 2 e_2 + e_3 + e_4$; \,
$\mathrm{dim}(L^{\left[0,1,0,2\right]})=2772$.
The set of singular weights for $B_4$ contains 384 vectors.
Here the defect is nontrivial, $\mathcal{D}_{\af_{\perp }}=-2\left( e_1 + e_2 \right)$,
while $\hf_{\bot}=0$.
Taking this into account we find among the singular weights
48 vectors with the property $\left\{ \mu _{\af_{\perp }%
}\left( u\right) =\pi _{\af_{\perp }}\left[ u(\mu +\rho
)-\rho \right] -\mathcal{D}_{\af_{\perp }}\in \overline{C_{\af_{\perp }}}\right\} $,
scalar products of these weights with all the roots in $\Delta^{+}_{\afb}$ are nonnegative.
The set $U=\left\{ u \right\}$ is thus fixed.
Compute the dimensions of the corresponding $\afb$-modules with the highest weights
$ \mu _{\af_{\perp }}\left( u\right)$ (using the Weyl dimension formula) and multiply them
by $\epsilon\left( u \right)$. The result is the singular element
$\Psi ^{\left( \mu \right) }_{\left(  \af, \afb \right)}$ shown in Figure \ref{fig:B4B2anom}.

Now one can place the fan $\Gamma$ from Figure \ref{fig:B4B2Fan} in the highest of the weights
presented in Figure \ref{fig:B4B2anom} and start the recursive determination of the branching coefficients
(using relation (\ref{recurrent-relation})):
\begin{eqnarray*}
  \label{eq:24}
  \pi_{\af} \left(ch L^{\left[0,1,0,2\right]}_{B_4}\right) = 6 \; ch L^{\left[0,0\right]}_{B_2}+ 60
  \; ch L_{B_2}^{\left[0,2\right]}+ 30 \; ch L_{B_2}^{\left[1,0\right]}+ 19 \; ch L_{B_2}^{\left[2,0\right]}+\\
  40 \; ch L_{B_2}^{\left[1,2\right]}+ 10 \; ch L_{B_2}^{\left[0,4\right]}.
\end{eqnarray*}
%\newpage
\section{Applications to the conformal field theory}
\label{sec:phys-appl}

\subsection{Conformal embeddings}
\label{sec:conformal-embeddings}

Branching coefficients for an embedding of affine Lie algebra into
affine Lie algebra can be used to construct modular invariant
partition functions for Wess-Zumino-Novikov-Witten models in conformal field theory
(\cite{difrancesco1997cft}, \cite{Walton:1999xc}, \cite{walton1989conformal}, \cite{schellekens1986conformal}).
In these models current algebras are affine Lie algebras.

The modular invariant partition function is crucial for the conformal theory to be valid
on the torus and higher genus Riemann surfaces. It is important for the applications of
CFT to string theory and to critical phenomena description.

The simplest modular-invariant partition function has the diagonal form:
\begin{equation}
  \label{eq:34}
   Z(\tau)=\sum_{ \mu\in P^{+}_{\mathfrak{g}}} \chi_{\mu}(\tau)\bar \chi_{\mu}(\bar \tau)
\end{equation}
Here the sum is over the set of the highest weights of integrable modules in a WZW-model
and $\chi_{\mu}(\tau)$ are the normalized characters (see \cite{difrancesco1997cft}) of these modules.

To construct the nondiagonal modular invariants is not an easy problem,
although for some models the complete classification of modular invariants is known \cite{1994hepthGannon,1995JMPGannon}.

Consider the Wess-Zumino-Witten model with the affine Lie algebra $\af$.
Nondiagonal modular invariants for this model can be constructed from the diagonal
invariant if there exists an affine algebra $\mathfrak{g}$ such that $\af\subset\mathfrak{g}$.
Then we can replace the characters of the $\mathfrak{g}$-modules in the diagonal
modular invariant partition function (\ref{eq:34})
by the decompositions
\begin{equation*}
  \label{eq:32}
\sum_{\nu \in P^{+}_{\af}}b^{(\mu)}_{\nu} \chi_{\nu}
\end{equation*}
containing normalized characters $\chi_{\nu}$ of the corresponding $\af$-modules.
Thus we obtain a nondiagonal modular-invariant  partition function for the theory with
the current algebra $\af$,
\begin{equation}
  \label{eq:36}
   Z_{\af}(\tau)=\sum_{ \nu,\lambda\in P^{+}_{\af}} \chi_{\nu}(\tau)M_{\nu\lambda}\bar \chi_{\lambda}(\bar \tau).
\end{equation}

The effective reduction procedure is crucial for this construction.
The embedding is required to preserve the conformal invariance.
Let $X^{\alpha_j}_{-n_j}$ and $\tilde{X}^{\alpha'_j}_{-n_j}$ be the lowering generators for
$\mathfrak{g}$ and for $\af\subset\mathfrak{g}$ correspondingly.
Let $\pi_{\af}$ be the projection operator of
$\pi_{\af}:\mathfrak{g}\longrightarrow \af$.
In the theory attributed to $\mathfrak{g}$ with the vacuum $\left|\lambda\right>$
the states can be described as
\begin{equation*}
  \label{eq:109}
  X^{\alpha_1}_{-n_1}X^{\alpha_2}_{-n_2}\dots\left|\lambda\right>\quad n_1\geq n_2\geq \dots>0.
\end{equation*}
And for the sub-algebra $\af$ the corresponding states are
\begin{equation*}
  \label{eq:110}
  \tilde{X}^{\alpha'_1}_{-n_1}\tilde{X}^{\alpha'_2}_{-n_2}\dots\left|\pi_{\af}(\lambda)\right>.
\end{equation*}
The $\mathfrak{g}$-invariance of the vacuum entails its $\af$-invariance,
but this is not the case for the energy-momentum tensor. So the energy-momentum tensor of the larger theory
should contain only the generators $\tilde{X}$. Then the relation
\begin{equation}
  \label{eq:2}
  T_{\mathfrak{g}}(z)=T_{\af}(z)
\end{equation}
leads to the equality of central charges
\begin{equation*}
  \label{eq:33}
  c(\mathfrak{g})=c(\af)
\end{equation*}
and to the relation
\begin{equation}
  \label{eq:111}
  \frac{k\;\mathrm{dim}\,\mathfrak{g}}{k+g}=\frac{x_e k\; \mathrm{dim}\,\af}{x_ek+a}.
\end{equation}
Here $x_e$ is the so called "embedding index":
$x_e=\frac{\left|\pi_{\mathfrak{a}} \Theta\right|^2}{\left|\Theta_{\mathfrak{a}}\right|^2}$
with $\Theta$, $\Theta_{\mathfrak{a}}$ being the highest roots of
$\mathfrak{g}$ and $\mathfrak{a}$
while $g$  and $a$ are the  corresponding dual Coxeter numbers.

It can be demonstrated that solutions of equation (\ref{eq:111}) exist only
for the level $k=1$ \cite{difrancesco1997cft}.

The complete classification of conformal embeddings is given in \cite{schellekens1986conformal}.

The relation (\ref{eq:111}) and the asymptotics of the branching functions can be used
to prove the finite reducibility theorem \cite{kac1988modular}.
It states that for a conformal embedding  $\af\longrightarrow\mathfrak{g}$
only finite number of branching coefficients have nonzero values.

\begin{mynote} The orthogonal subalgebra $\afb$ is always trivial
for conformal embeddings $\af\longrightarrow \mathfrak{g}$.
\begin{proof}
Consider the modes expansion of the energy-momentum tensor
\begin{equation*}
\label{eq:47}
  T(z)=\frac{1}{2(k+h^v)}\sum_n z^{-n-1}L_n.
\end{equation*}
The modes $L_n$ are constructed as combinations of normally-ordered products of the generators of $\mathfrak{g}$,
\begin{equation*}
\label{eq:48}
  L_n=\frac{1}{2(k+h^v)}\sum_{\alpha}\sum_m:X^{\alpha}_m X^{\alpha}_{n-m}: \; .
\end{equation*}
In case of a conformal embedding energy-momentum tensors $T_{\mathfrak{g}}(z)$ and $T_{\af}(z)$ are
equal (see (\ref{eq:2})).

Substituting generators of $\af$  in terms of generators of $\mathfrak{g}$ into these combinations
we must obtain the energy-momentum tensor $T_{\mathfrak{g}}$.
But if the set of generators attributed to $\Delta_{\afb}$ is not empty this is not possible,
since $T_{\mathfrak{g}}$ contains generators $X^{\alpha}_n$ for $\alpha\in \Delta_{\afb}$.
\end{proof}
\end{mynote}

\subsubsection{Special embedding $\hat{A}_1\rightarrow\hat{A}_2$.}
\label{sec:spec-embedd-hata_1s}

Consider the case where both $\gf$ and $\af$ are affine Lie algebras:
$\hat{A}_1 \rightarrow \hat{A}_2$ and the injection is the affine extension of the
special injection $A_1 \rightarrow A_2$ with the embedding index $x_e=4$.
As far as the $\gf$-modules to be considered are of level one,
the $\af$-modules will be of level $\tilde{k}=kx_e=4$.

There exist three level one fundamental weights of $\hat{A}_2$.
It is easy to see that the set $\Delta_{ \afb }$ is empty and the subalgebra $\afb=0$. Then ${\cal D}_{\afb}=0$, $\hf_{\perp}$ is one-dimensional Abelian subalgebra and the dimensions of $\tilde\afb=\afb\oplus \hf_{\perp}$ are equal to 1.
It is convenient to choose the classical root for $\hat{A}_1$ to be
$\beta=\frac{1}{2}(\alpha_1+\alpha_2)$.

Using Definition (\ref{fan-definition})  we construct the fan $\Gamma_{\hat A_1\to\hat A_2}$.
In this case $\gamma_0 =0$ and its sign $s\left( 0 \right)=-1$ thus we are to use the
sign function $s(\gamma)$ (see Figure \ref{fig:AffineA2A1Fan}).

\begin{figure}[h!bt]
  \centering
  \includegraphics[width=125mm]{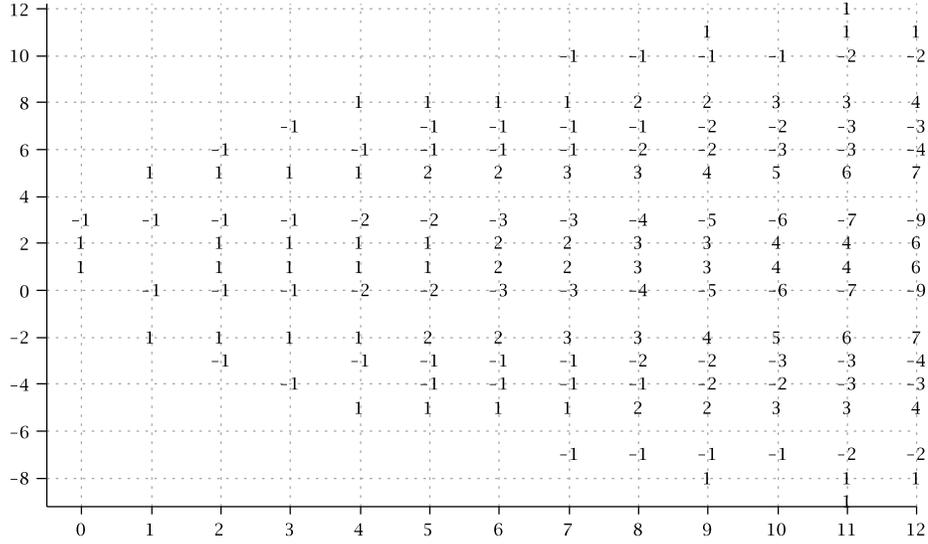}

  \caption{The fan $\Gamma_{\hat{A_1}\rightarrow \hat{A_2}}$ for $\hat{A_1}\rightarrow \hat{A_2}$
  in the basis $\left\{\beta,\delta \right\}$. Notice that $\gamma_0 =0$, so values of $s(\gamma)$
  are prescribed to the weights $\gamma\in \Gamma_{\hat{A_1}\rightarrow \hat{A_2}}$}
  \label{fig:AffineA2A1Fan}
\end{figure}

Consider the module $L^{\omega_0=(0,0;1;0)}$. Here we use the (finite part; level; grade)
presentation of the highest weight and the finite part
coordinates are the Dynkin indices (see section(\ref{sec:notation})).

The set $\widehat{\Psi^{(\omega_0)}}$  is displayed in Figure
\ref{fig:affine_A2_anom_point} up to the sixth grade.

\begin{figure}[h!tb]
  \hspace*{-1.5cm}
  \includegraphics[width=180mm]{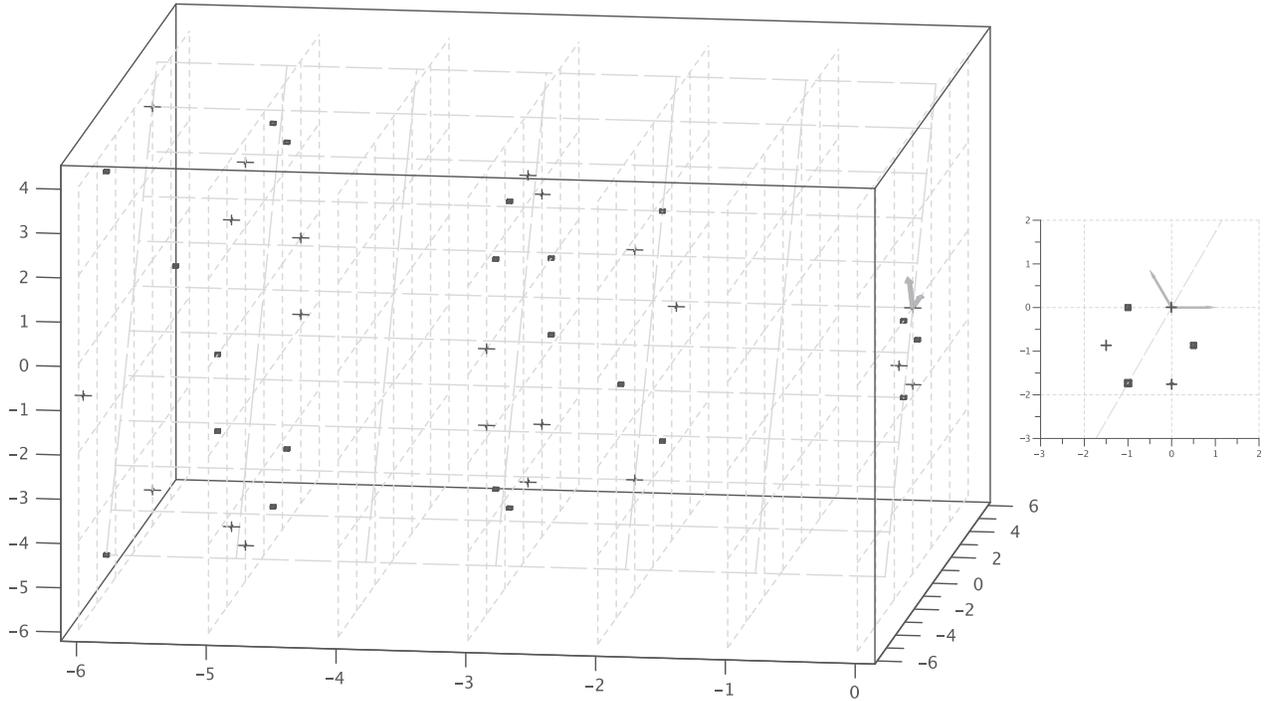}
  \caption{The singular weights of the module $L_{\hat{A_2}}^{\omega_0}=L^{(0,0;1;0)}_{\hat{A_2}}$.
   The classical (grade zero) cross-section of the diagram is shown separately
   in the right part of the figure.
  We use the orthogonal basis with the unit vector equal to $\alpha_1$.
  The weights $w (\omega_0+\rho)-\rho$ are marked by crosses when $\epsilon(w)=1$ and
by box when $\epsilon(w)=-1$. Simple roots of the classical subalgebra $A_2$ are
grey and the grey diagonal plane corresponds to the Cartan subalgebra of
the embedded algebra $\hat{A}_1$.}
  \label{fig:affine_A2_anom_point}
\end{figure}

The next step is  to project the anomalous weights to $P_{\hat A_1}$.
The result is the element $\Psi ^{\left( \omega_0 \right) }_{\left(  \hat A_1\, , \, \afb=0 \right)}$
presented in Figure \ref{fig:AffineA2_A1_anom_proj} up to the twelfth grade.
\begin{figure}[h!tb]
  \centering
  \includegraphics[width=130mm]{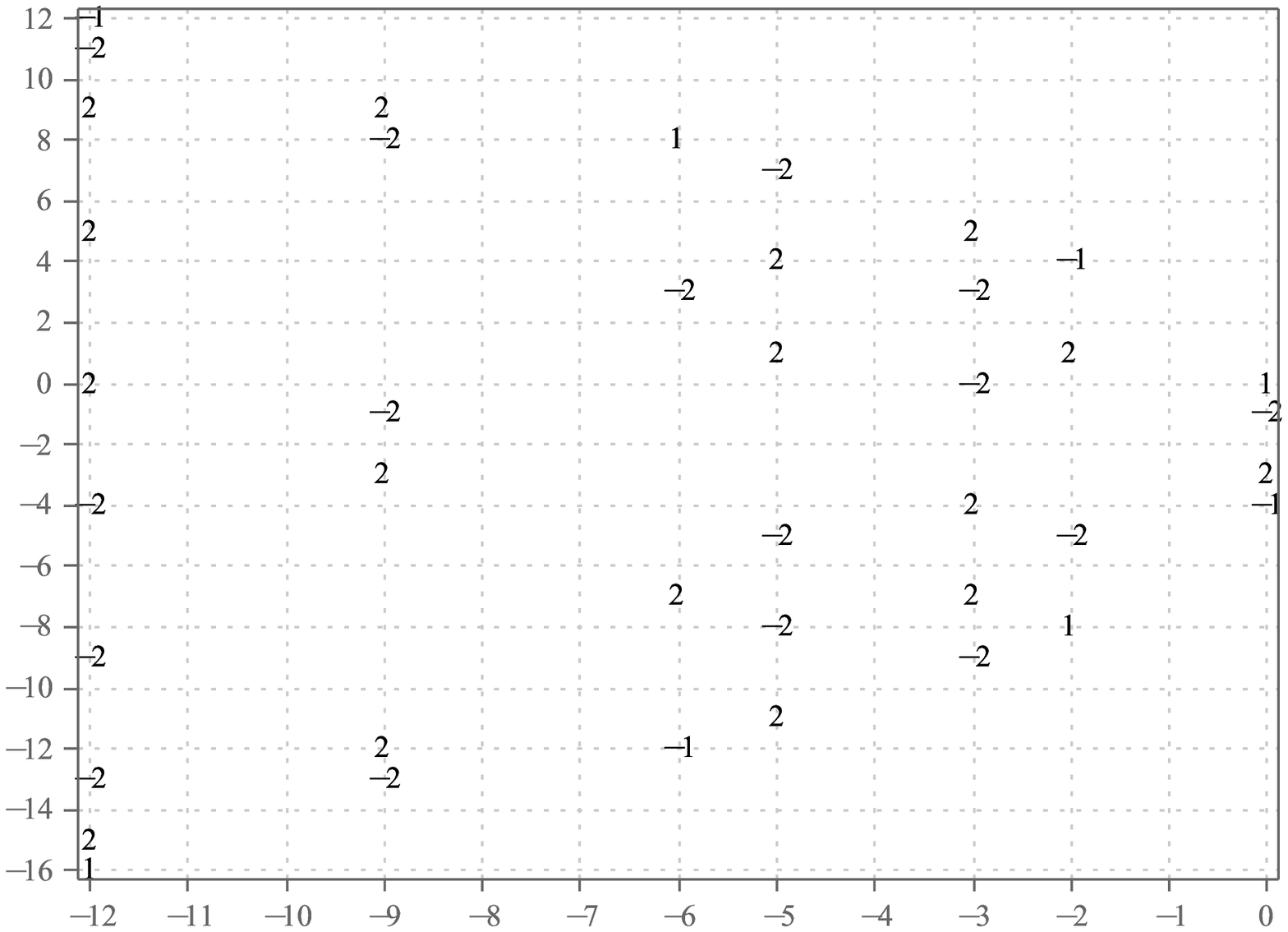}
  \caption{The singular element $\Psi ^{\left( \omega_0 \right) }_{\left(  \hat A_1\, , \, \afb=0 \right)}$
  displayed in $P_{\hat A_1}$
  with the basis $\left\{\beta,\delta \right\}$.}
  \label{fig:AffineA2_A1_anom_proj}
\end{figure}

Using the recurrent relation (\ref{recurrent-relation}) with the fan
$\Gamma_{\hat{A_1}\rightarrow \hat{A_2}}$ and the singular weights in
$\Psi ^{\left( \omega_0 \right) }_{\left(  \hat A_1\, , \, \afb=0 \right)}$
we get the anomalous branching coefficients presented
in Figure \ref{fig:AffineA2_A1_branching}.
\begin{figure}[h!tb]
  \centering
  \includegraphics[width=130mm]{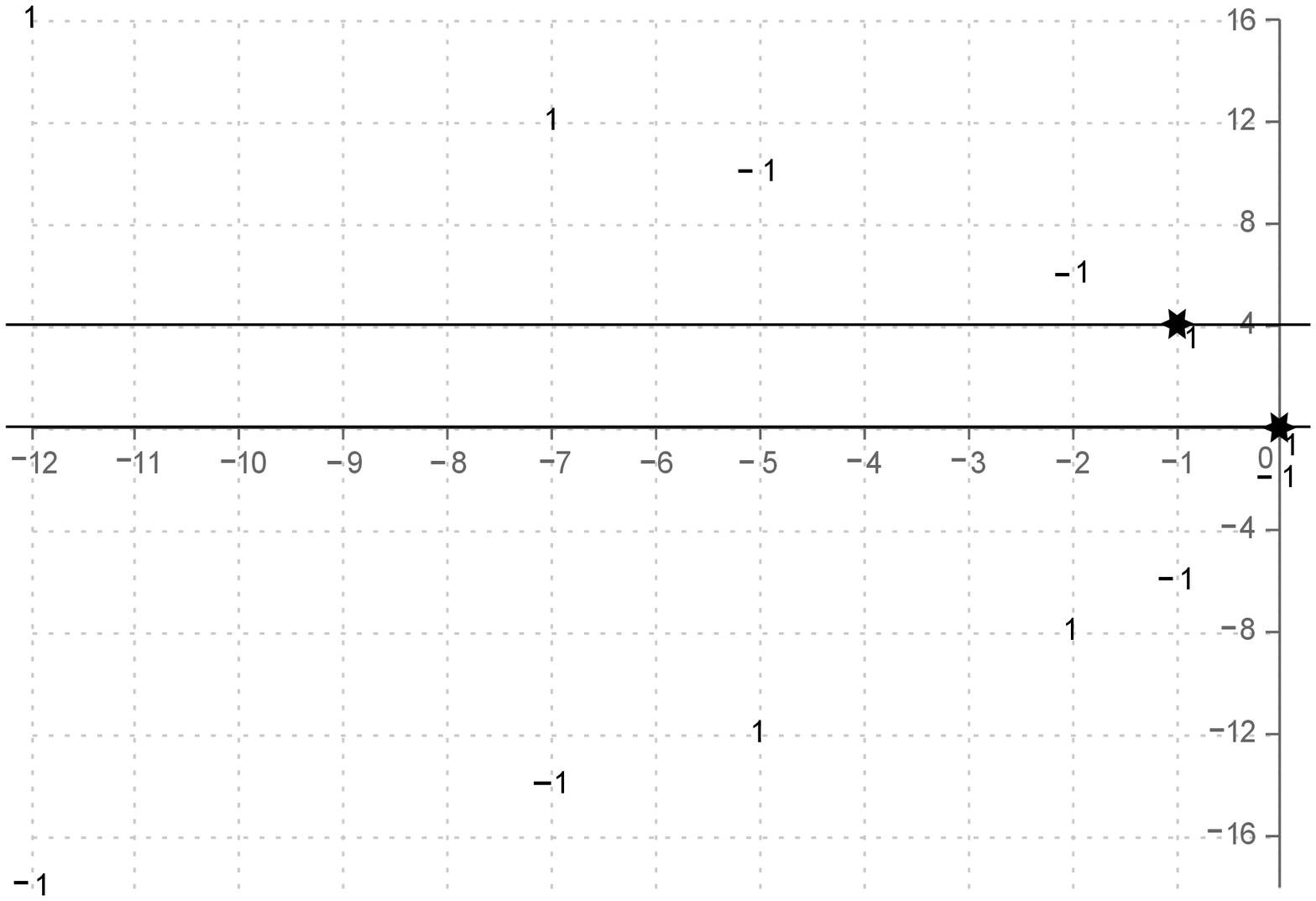}
  \caption{Anomalous branching coefficients for $\hat{A_1}\subset \hat{A_2}$. The boundaries
  of the main Weyl chamber $\bar{C}_{\hat{A}_1}$
 are indicated by black lines. Two anomalous highest weights located
 in the main Weyl chamber are marked by stars.
 Both have multiplicity 1, so the branching coefficients for them are equal 1.}
  \label{fig:AffineA2_A1_branching}
\end{figure}
Inside the Weyl chamber $\bar{C}_{\hat{A}_1}$
(its boundaries are indicated in Figure \ref{fig:AffineA2_A1_branching})
there are only two nonzero anomalous weights and both have multiplicity 1.
These are the highest weights of $\af$-submodules and the multiplicities are their branching
coefficients. Thus we get the decomposition
\begin{equation*}
  \label{eq:43}
  L^{(0,0;1;0)}_{\hat{A_2}\downarrow \hat{A_1}}= L_{\hat{A_1}}^{(0;4;0)}\oplus L_{\hat{A_1}}^{(4;4;0)}.
\end{equation*}
The finite reducibility theorem holds.

The same fan $\Gamma_{\hat{A_1}\rightarrow \hat{A_2}}$ can be used for any other highest weight
module $L^{\mu}_{\hat{A_2}}$. In particular for the irreducible modules of level one  we get the trivial
branching:
\begin{equation*}
  \label{eq:44}
   L^{(1,0;1;0)}_{\hat{A_2}\downarrow \hat{A_1}}= L_{\hat{A_1}}^{(2;4;0)},\\
   L^{(0,1;1;0)}_{\hat{A_2}\downarrow \hat{A_1}}= L_{\hat{A_1}}^{(2;4;0)}.
\end{equation*}

Using these results the modular-invariant partition function is easily found,
\begin{equation*}
  \label{eq:45}
  Z=\left|\chi_{(4;4;0)}+\chi_{(0;4;0)}\right|^2+2\chi_{(2;4;0)}^2.
\end{equation*}

\subsection{Coset models}
\label{sec:coset-models}

Coset models \cite{Goddard198588} tightly connected with the gauged WZW-models are actively studied
in string theory, especially in string models on anti-de-Sitter space
\cite{Maldacena:2000hw,Maldacena:2000kv,Maldacena:2001km,Maldacena:2001ky,Aharony:1999ti}.
The characters in coset models are proportional to branching functions,
\begin{equation}
  \label{eq:31}
  \chi^{(\mu)}_{\nu}(\tau)=e^{2\pi i \tau (m_{\mu}-m_{\nu})} b^{(\mu)}_{\nu}(\tau),
\end{equation}
with
\begin{equation*}
  \label{eq:46}
  m_{\mu}=\frac{\left|\mu+\rho\right|^2}{2(k+g)}-\frac{\left|\rho\right|^2}{2g}.
\end{equation*}
The problem of branching functions construction in the coset models was considered
in  \cite{Dunbar:1992gh}, \cite{Hwang:1994yr}, \cite{lu1994branching}.

Let us return to the example \ref{sec:regul-embedd-a_1} and consider the affine extension of the injection
$A_1 \rightarrow B_2$.
Since this embedding is regular and $x_e=1$, the subalgebra modules and the initial module are of the same level.
The set  of positive roots with zero projection
on the root space of the subalgebra $\hat{A_1}$ is the same as in the finite-dimensional case
$\Delta^{+}_{\afb}=\left\{ \alpha_1 \right\}$ and $\afb=A_1$. It is easy to see that $\hf_{\perp}$
is trivial in this case and also ${\cal D}_{\afb}=0$.

Using the definition (\ref{fan-defined}) we obtain the fan
$\Gamma_{\hat{A_1} \longrightarrow  \hat{B_2} }$.
Notice that here the lowest weight  $\gamma_0$ of the fan  is zero and $s\left( \gamma_0 \right)=-1$.
The values of the sign function $s(\gamma)$ for
$ \gamma \in \Gamma_{\hat{A_1} \longrightarrow  \hat{B_2} }$ are presented in Figure \ref{fig:AffineB2A1Fan}.
We restricted the computation to the twelfth grade.
\begin{figure}[h!bt]
  \centering
  \includegraphics[width=135mm]{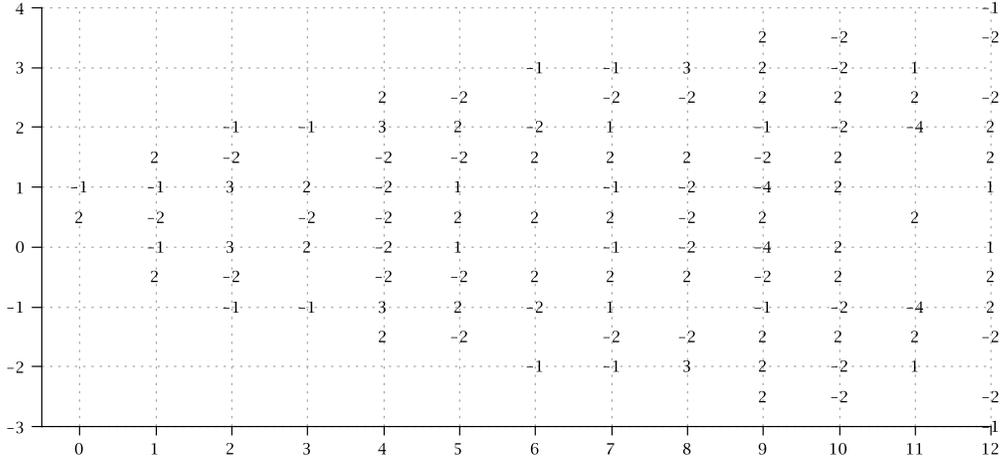}
  \caption{The fan $\Gamma_{\hat{A_1}\rightarrow \hat{B_2}}$
  for $\hat{A_1}\rightarrow \hat{B_2}$ in the basis $\left\{\beta,\delta \right\}$. Values of  $s(\gamma)$ are shown for the
  weights $\gamma\in \Gamma_{\hat{A_1}\rightarrow \hat{B_2}}$}
  \label{fig:AffineB2A1Fan}
\end{figure}

Consider the level one module $L^{\left( 1,0;1;0 \right)}_{\hat{B_2}}$  with the highest weight $\omega_1=(1,0;1;0)$,
where the finite part coordinates are in the orthogonal basis $e_1,e_2$.
The set of anomalous weights for this module up to the sixth grade is presented in the Figure \ref{fig:affine_B2_anom_point}.

\begin{figure}[h!tb]
%  \hspace*{-2cm}
  \includegraphics[width=140mm]{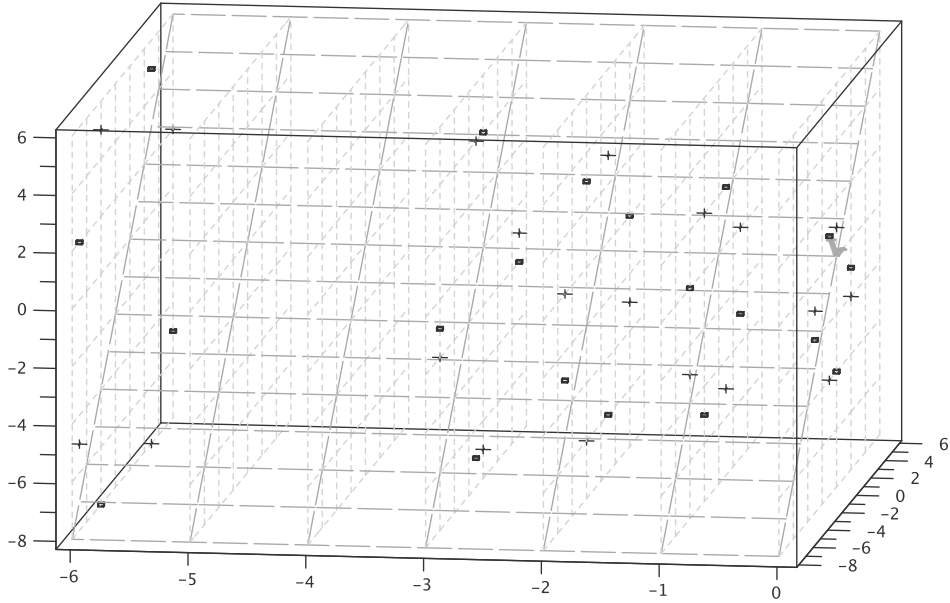}
  \caption{Singular weights for $L^{(1,0;1;0)}_{\hat B_2 }$. The standard basis $\{e_1,e_2\}$ is used for the classical cross-section.
  The weights in the zero grade are the same as in Figure \ref{fig:B2_A1}.
  The weights $w (\omega_1+\rho)-\rho$ are marked by crosses if $\epsilon(w)=1$ and by boxes for $\epsilon(w)=-1$.
Simple roots of the classical subalgebra $B_2$ are grey and grey diagonal plane corresponds to the Cartan subalgebra
of the embedded algebra $\hat{A}_1$.}
  \label{fig:affine_B2_anom_point}
\end{figure}

According to the algorithm \ref{sec:algorithm} we project the anomalous weights to
$P_{\hat{A_1}}$ and find the dimensions of the corresponding
$\afb$-modules $L^{\pi_{\afb}(w(\mu+\rho))-\rho_{\afb}}_{\afb}$.
In the grade zero this projection gives exactly the set
$\Psi ^{\left( \mu \right) }_{\left(  A_1, A_1 \right)}$ for the embedding of
the classical Lie algebra $A_1\rightarrow B_2$. To see this compare Figure \ref{fig:B2_A1}
with the Figure \ref{fig:AffineB2_A1_anom_proj}
where the singular element $\Psi ^{\left( \mu \right) }_{\left(  \widehat{A_1}, A_1 \right)}$
for the affine embedding $\hat{A_1}$ is presented up to the twelfth grade.
\begin{figure}[h!tb]
  \centering
  \includegraphics[width=120mm]{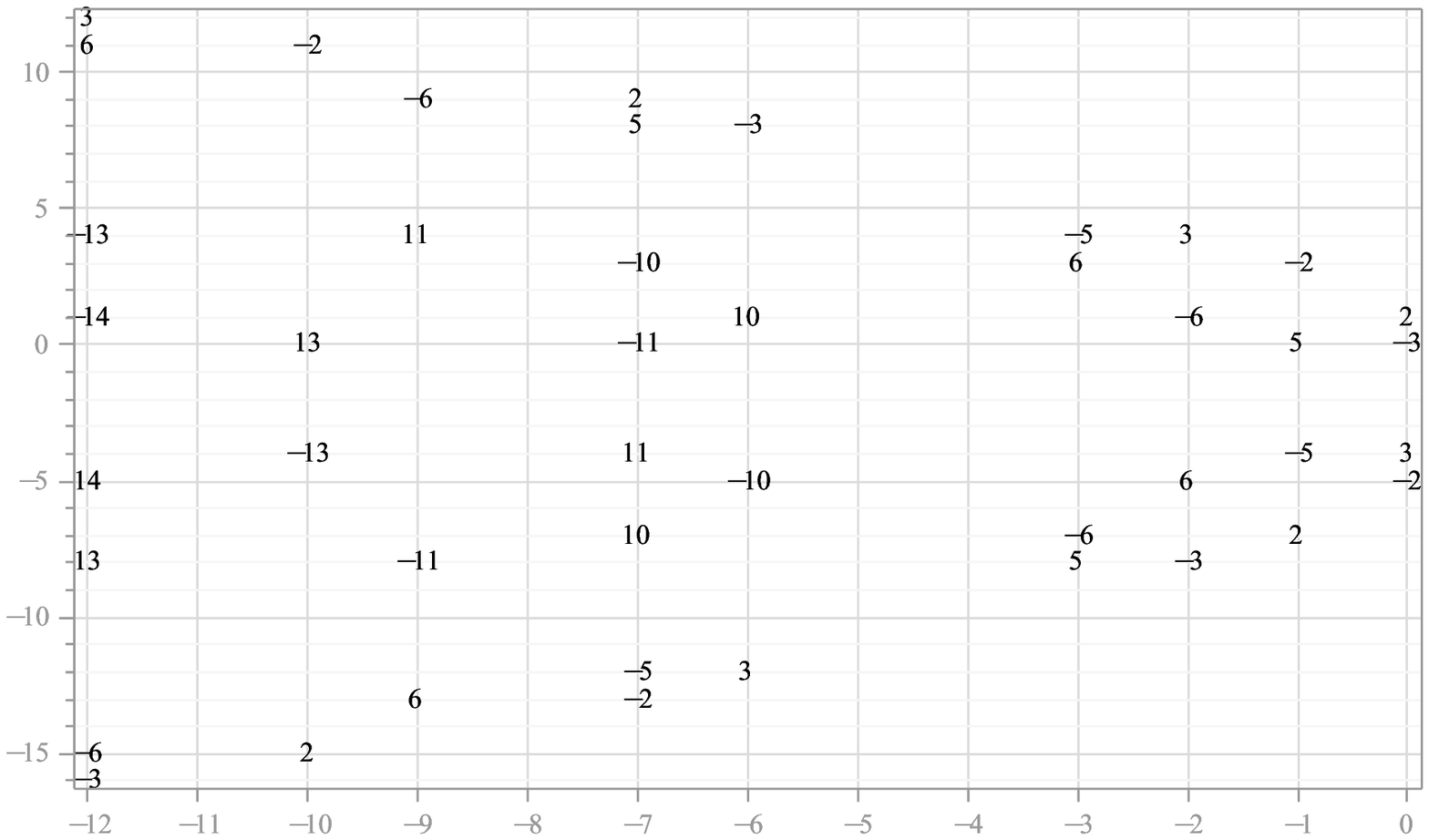}
  \caption{The singular element $\Psi ^{\left( \omega_1 \right) }_{\left(  \widehat{A_1}, A_1 \right)}$
  in the basis $\{\beta,\delta\}$. The dimensions of the corresponding $\afb=A_1$-modules
  with the signs $\epsilon(u)$ are indicated.}
  \label{fig:AffineB2_A1_anom_proj}
\end{figure}

\begin{figure}[h!bt]
  \centering
  \includegraphics[width=120mm]{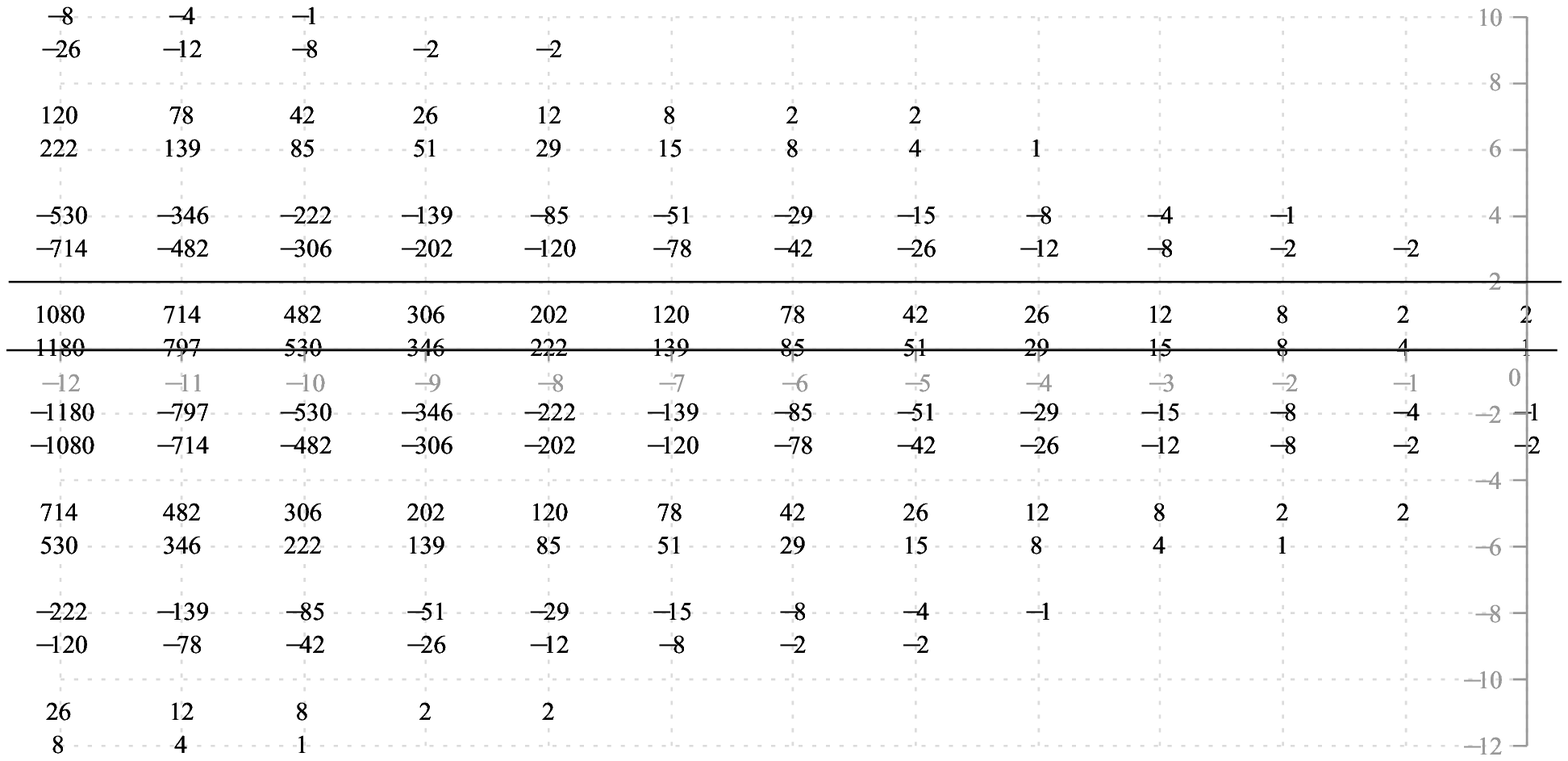}
  \caption{Anomalous branching coefficients for $\hat{A_1}\rightarrow \hat{B_2}$. The basis $\{\beta,\delta\}$ is used.
  The boundaries  of the main Weyl chamber $\bar{C}_{\hat{A}_1}$
 are indicated by the black lines. The anomalous branching coefficients
 inside the main Weyl chamber are equal to the branching coefficients of the embedding $\hat{A_1}\rightarrow \hat{B_2}$.}
  \label{fig:AffineB2_A1_branching}
\end{figure}

The multiplicities of the highest weights inside the  Weyl chamber
$\bar{C}^{\left( 0 \right)}_{\hat{A_1}}$
define the following branching coefficients (up to the twelfth grade),
\begin{eqnarray*}
  \label{eq:28}
  L^{\omega_1}_{\hat{B_2}\downarrow \hat{A_1}}
  &=&2 L_{\hat{A_1}}^{\omega_1}\oplus 1 L_{\hat{A_1}}^{\omega_0}\oplus 4 L_{\hat{A_1}}^{\omega_0-\delta}\oplus\\
    &&2 L_{\hat{A_1}}^{\omega_1-\delta}\oplus 8 L_{\hat{A_1}}^{\omega_0-2\delta}\oplus
    8 L_{\hat{A_1}}^{\omega_1-2\delta}\oplus 15 L_{\hat{A_1}}^{\omega_0-3\delta}\oplus\\
    &&12 L_{\hat{A_1}}^{\omega_1-3\delta}\oplus 26 L_{\hat{A_1}}^{\omega_1-4\delta}\oplus
    29 L_{\hat{A_1}}^{\omega_0-4\delta}\oplus 51 L_{\hat{A_1}}^{\omega_0-5\delta}\oplus\\
    &&42 L_{\hat{A_1}}^{\omega_1-5\delta}\oplus 78 L_{\hat{A_1}}^{\omega_1-6\delta}\oplus
    85 L_{\hat{A_1}}^{\omega_0-6\delta}\oplus 120 L_{\hat{A_1}}^{\omega_1-7\delta}\oplus\\
    &&139 L_{\hat{A_1}}^{\omega_0-7\delta}\oplus 202 L_{\hat{A_1}}^{\omega_1-8\delta}\oplus
    222 L_{\hat{A_1}}^{\omega_0-8\delta}\oplus 306 L_{\hat{A_1}}^{\omega_1-9\delta}\oplus\\
    &&346 L_{\hat{A_1}}^{\omega_0-9\delta}\oplus 530 L_{\hat{A_1}}^{\omega_0-10\delta}\oplus
    482 L_{\hat{A_1}}^{\omega_1-10\delta}\oplus 714 L_{\hat{A_1}}^{\omega_1-11\delta}\oplus\\
    &&797 L_{\hat{A_1}}^{\omega_0-11\delta}\oplus 1080 L_{\hat{A_1}}^{\omega_1-12\delta}\oplus
    1180 L_{\hat{A_1}}^{\omega_0-12\delta}\oplus \dots
\end{eqnarray*}
This result can be presented as the set of branching functions:
\begin{eqnarray*}
  \label{eq:29}
  \begin{array}{cc}
    b^{(\omega_1)}_{0}= & 1 + 4\,q^{1}+ 8\,q^{2}+ 15\,q^{3}+ 29\,q^{4}+ 51\,q^{5}+ 85\,q^{6}+ 139\,q^{7}+\\
     &222\,q^{8}+ 346\,q^{9}+ 530\,q^{10}+ 797\,q^{11}+ 1180\,q^{12}+\dots\\
  \end{array}\\
  \begin{array}{cc}
    b^{(\omega_1)}_{1}= &2+2\,q^{1}+8\,q^{2}+12\,q^{3}+26\,q^{4}+42\,q^{5}+78\,q^{6}+120\,q^{7}+\\
    & 202\,q^{8}+306\,q^{9}+482\,q^{10}+714\,q^{11}+1080\,q^{12}+\dots
  \end{array}
\end{eqnarray*}
Here $q=\exp (2\pi i \tau)$ and the lower index enumerates the branching functions according
to their highest weights in $P^+_{\hat{A_1}}$.
These are the fundamental weights $\omega_0=\lambda_0=(0,1,0),\; \omega_1=\alpha/2=(1,1,0)$.

Now we can use the relation (\ref{eq:31}),
\begin{equation*}
  \label{eq:35}
  \begin{array}{cc}
    \chi^{(\omega_1)}_{1}(q)= & q^{\frac{7}{12}}\left( 2+2\,q^{1}+8\,q^{2}+12\,q^{3}+26\,q^{4}+42\,q^{5}+78\,q^{6}+120\,q^{7}+\right. \\
    & \left. 202\,q^{8}+306\,q^{9}+482\,q^{10}+714\,q^{11}+1080\,q^{12}+\dots \right),\\
    \chi^{(\omega_1)}_{0}(q) = & q^{\frac{5}{6}}\left(1 + 4\,q^{1}+ 8\,q^{2}+ 15\,q^{3}+ 29\,q^{4}+ 51\,q^{5}+ 85\,q^{6}+ 139\,q^{7}+\right. \\
    &\left. 222\,q^{8}+ 346\,q^{9}+ 530\,q^{10}+ 797\,q^{11}+ 1180\,q^{12}+\dots\right),
  \end{array}
\end{equation*}
and thus obtain the expansion of the $B_2/A_1$-coset characters.

\section{Conclusion}
\label{sec:conclusion}
We have demonstrated that the injection fan technique can be used to deal with an arbitrary
reductive subalgebra (maximal as well as  nonmaximal).
It was shown that the branching problem for $\af \subset \gf$  is tightly connected with
the properties of the orthogonal partner $ \af_{\perp
} $ of $\af$. The subalgebra $\afb$ corresponds to the subset
$\Delta^{+}_{\afb}$ of positive roots in $\Delta_{\mathfrak{g}}^{+}$ that trivialize
the Cartan subalgebra $\hf_{\afb}$.
Both the injection fan and the sets of singular weights for
the highest weight $\gf$-modules depend substantially on the structure of $\afb$ and its submodules.
For the fan $\Gamma_{\af\rightarrow \gf}$ this dependence is almost obvious:
in the element $\Phi_{\af\rightarrow \gf}$ the factors corresponding to the roots
of $\Delta^{+}_{\afb}$ are eliminated.
The transformation in the set of projected singular weights is more interesting.
We have found out that in the new singular element
$\Psi ^{\left( \mu \right) }_{\left(  \af, \afb \right)}$ the coefficients depend on the
the $\afb$-submodules (their highest weights $\mu _{\widetilde{\af_{\perp }}}\left( u\right)$
are fixed by the injection and by the weights of the initial
element $\Psi^{\mu}$).
Fortunately no more information on $L^{\mu _{\widetilde{\af_{\perp }}}\left( u\right)}
_{\afb}$-submodules is necessary than their dimensions.
In the new singular element $\Psi ^{\left( \mu \right) }_{\left(  \af, \afb \right)}$
the weight multiplicities are equal to the dimensions
$\dim\left(L^{\mu _{\widetilde{\af_{\perp }}}\left( u\right)}_{ \afb }\right)$
of the corresponding $\afb$-modules
multiplied by the values $\epsilon (u)$. As a result
the highest weights of $\af$-submodules and their
multiplicities are subject to the set of linear equations (\ref{eq:17}).
These properties are valid for any reductive subalgebra $\af\rightarrow \gf$ and
the set can be redressed to the form of recurrent relations to be solved step by step.

The efficiency of the obtained algorithm was illustrated in various examples.
In particular we considered the construction of modular-invariant partition functions
in the framework of conformal embedding method and the coset construction in the rational conformal field theory.
This construction is useful in the study of WZW-models
emerging in the context of the AdS/CFT correspondence \cite{Maldacena:2000hw,Maldacena:2000kv,Maldacena:2001km}.

Further amelioration of the algorithm can be achieved by using
the folded fan technique \cite{il2010folded}. It must be mentioned that even in the case
of string functions the explicit solution of the corresponding recurrent relations is
a difficult problem (see \cite{il2010folded} for details). Nevertheless we hope that by
developing the procedure of folding one could get explicit solutions
for at least some of branching functions and the corresponding coset characters.

\section{Acknowledgements}
The work was supported in
part by RFFI grant N 09-01-00504 and the National Project RNP.2.1.1./1575.

\section*{References}
\bibliography{article}{}
\bibliographystyle{iopart-num}

\end{document}